\documentclass[12pt,letterpaper]{article}
\usepackage{amsthm,amssymb}
\usepackage{amsfonts}
\makeatletter
\usepackage[dvips]{color}
\usepackage{colordvi,multicol}

\newtheorem{thm}{\textbf Theorem}[section]
\newtheorem{lem}[thm]{\textbf Lemma}
\newtheorem{rem}[thm]{\textbf Remark}
\newtheorem{cor}[thm]{\textbf Corollary}

\newtheorem{defin}[thm]{\textbf Definition}
\newtheorem{assum}[thm]{\textbf Assumption}

\newcommand{\be}{\begin{eqnarray*}}
\newcommand{\ee}{\end{eqnarray*}}

\begin{document}

\title{\bf L\'evy-Khintchine type
representation of Dirichlet generators\\ and Semi-Dirichlet forms}
  \author{Wei Sun,\ \ \ \ Jing Zhang
            \\ Department of Mathematics and Statistics\\
             Concordia University\\
             Montreal, H3G 1M8, Canada\\
             wei.sun@concordia.ca (W. Sun),\\
              waangel520@gmail.com (J. Zhang)}
   \date{}
\maketitle

\begin{abstract}
\noindent Let $U$ be an open set of $\mathbb{R}^n$, $m$ a positive
Radon measure on $U$ such that ${\rm supp}[m]=U$, and
$(P_t)_{t>0}$ a strongly continuous contraction sub-Markovian
semigroup on $L^2(U;m)$. We investigate the structure of
$(P_t)_{t>0}$.
\begin{enumerate}\item[(i)] Denote  respectively by $(A,D(A))$ and $(\hat A,D(\hat A))$ the generator and the co-generator
of $(P_t)_{t>0}$. Under the assumption that $C^{\infty}_0(U)\subset D(A)\cap D(\hat A)$, we give an explicit L\'evy-Khintchine type representation of $A$ on $C^{\infty}_0(U)$.
\item[(ii)] If $(P_t)_{t>0}$ is an analytic semigroup and hence is associated with a semi-Dirichlet form $({\cal E}, D({\cal E}))$, we give an explicit characterization of ${\cal E}$  on $C^{\infty}_0(U)$ under the assumption that $C^{\infty}_0(U)\subset D({\cal E})$.
\end{enumerate}
We also present a LeJan type transformation rule for the diffusion part of
regular semi-Dirichlet forms on general state spaces. \vskip 0.5cm \noindent {\it
Keywords:} Dirichlet generator; Semi-Dirichlet form; Markov process; L\'evy-Khintchine type representation; LeJan type
transformation rule; Beurling-Deny formula\vskip 0.5cm \noindent

\end{abstract}

\section[short title]{Introduction and main results}

Let $(X_t)_{t\geq 0}$ be a L\'{e}vy process  on
$\mathbb{R}^n$. By the celebrated L\'{e}vy-Khintchine formula, we know that the infinitesimal generator $A$ of $(X_t)_{t\geq 0}$
is characterized by (cf. \cite[Theorem 31.5]{S99})
\begin{eqnarray}\label{decom23}
& &Au(y)=\frac{1}{2}\sum_{i,j=1}^nQ_{ij}\frac{\partial^2u}{\partial y_i\partial y_j}(y)+\sum_{i=1}^nb_i\frac{\partial u}{\partial y_i}(y)\nonumber\\
& &\ \ \ +\int_{\mathbb{R}^n}\left(
u(y+x)-u(y)-\sum_{i=1}^nx_i\frac{\partial u}{\partial
y_i}(y)I_{\{|x|\le 1\}}(x)\right)\nu(dx),\ \ u\in
C_0^{\infty}(\mathbb{R}^n),\ \ \ \ \ \ \ \ \ \
\end{eqnarray}
where $Q=(Q_{ij})_{1\le i,j\le n}$ is a symmetric
nonnegative-definite $n\times n$ matrix, $(b_1,\dots,b_n)\in
\mathbb{R}^n$, and $\nu$ is a L\'{e}vy measure satisfying
$\nu(\{0\})=0$ and $\int_{\mathbb{R}^n}(1\wedge
|x|^2)\nu(dx)<\infty$. Hereafter, $|\cdot|$ denotes the Euclidean
metric of $\mathbb{R}^n$, $C(\mathbb{R}^n)$ denotes the set of all
continuous functions on $\mathbb{R}^n$, and
$C_0^{\infty}(\mathbb{R}^n)$ denotes the set of all infinitely
differentiable functions on $\mathbb{R}^n$ with compact supports.

The decomposition of type (\ref{decom23}) also holds for Feller
processes on $\mathbb{R}^n$. In \cite{Cou},  Courr\`{e}ge proved that if $A$ is a linear operator
from $C_0^{\infty}(\mathbb{R}^n)$ to $C(\mathbb{R}^n)$ satisfying
the positive maximum principle, then $A$ is decomposed as
\begin{eqnarray*}
Au(y)&=&\frac{1}{2}\sum_{i,j=1}^nq_{ij}(y)\frac{\partial^2u}{\partial y_i\partial y_j}(y)+\sum_{i=1}^nl_i(y)\frac{\partial u}{\partial y_i}(y)+\gamma(y)u(y)\nonumber\\
& &+\int_{\mathbb{R}^n}\left(u(y+x)-u(y)w(x)-{\sum_{i=1}^nx_i\frac{\partial u}{\partial y_i}(y)}w(x)\right)\mu(y,dx),
\end{eqnarray*}
where $\sum_{i,j=1}^nq_{ij}(y)\xi_i\xi_j\ge 0$ for all $y\in
\mathbb{R}^n$ and $(\xi_1,\dots,\xi_n)\in \mathbb{R}^n$, the
function $y\rightarrow \sum_{i,j=1}^nq_{ij}(y)\xi_i\xi_j$ is upper
semicontinuous, $l_i\in C(\mathbb{R}^n)$, $1\le i\le n$,
$\gamma\in C(\mathbb{R}^n)$ with $\gamma\le 0$, $\mu$ is a kernel
on $\mathbb{R}^n\times {\cal B}(\mathbb{R}^n)$, and $w\in
C_0^{\infty}(\mathbb{R}^n)$ with $0\le w\le 1$ and $w=1$ on
$\{x\in \mathbb{R}^n:|x|\le 1\}$ (cf. \cite[\S 4.5]{J01}).

Suppose now that $(X_t)_{t\geq 0}$ is a general right continuous
Markov process on $\mathbb{R}^n$, or more generally, on an open
set $U$ of $\mathbb{R}^n$. In this paper, we are interested in describing the
analytic structure of $(X_t)_{t\geq 0}$. Denote by $(P_t)_{t>0}$ the
transition semigroup of $(X_t)_{t\geq 0}$. Suppose that there is a
positive Radon measure $m$ on $U$ such that $(P_t)_{t>0}$ acts as a
strongly continuous contraction semigroup on $L^2(U;m)$. Note that this condition is fulfilled if, for example, $m$ is an excessive measure
of $(X_t)_{t\geq 0}$. Denote by
$(A,D(A))$ the $L^2$-generator of $(P_t)_{t>0}$. Then, $(A,D(A))$ is a Dirichlet
operator, i.e., $(Au, (u-1)\vee 0)\le 0$ for all $u\in D(A)$ (cf.
\cite[Proposition I.4.3]{MR92}). Hereafter $(\cdot,\cdot)$ denotes
the inner product of $L^2(U;m)$.

Denote by $(\hat A,D(\hat A))$ the co-generator
of $(P_t)_{t>0}$. Note that generally $(\hat A,D(\hat A))$ may not be a Dirichlet
operator (see \cite[Remark 2.2(ii)]{MR95} for an example). We assume that $C^{\infty}_{0}(U)\subset D(A)\cap D(\hat A)$ and
consider the following bilinear form
\begin{equation}\label{cvbn}
{\cal E}(u,v):=(-Au,v)\ \ {\rm for}\ u,v\in C^{\infty}_{0}(U).
\end{equation}
Here we would like to remind the reader that a generator on an $L^2$-space is a Dirichlet
operator if and only if its associated semigroup is sub-Markovian. It does not imply that its associated bilinear form is a (pre-) semi-Dirichlet form since the sector condition might not be satisfied. Denote by $(G_{\beta})_{\beta>0}$ and $(\hat G_{\beta})_{\beta>0}$
the resolvent and co-resolvent of $(P_t)_{t>0}$, respectively.
Similar to \cite[\S 2]{HC06} (cf. also \cite[\S 3.2]{Fu94}) and noting that the sector condition is not used therein, we
can prove the following lemma by virtue of the fact that ${\cal
E}(u,v)=\lim_{\beta\rightarrow\infty}\beta(u-\beta G_{\beta}u,v)$ for $u,v\in C^{\infty}_{0}(U)$.
\begin{lem}\label{zxcv22}
(i) For $\beta>0$, there exist unique positive Radon measures
$\sigma_{\beta}$ and $\hat\sigma_{\beta}$ on $U\times U$
satisfying
\begin{equation}\label{E1}
(\beta G_{\beta}u,v)=\int_{U\times U}u(x)v(y)\sigma_{\beta}(dxdy)
\end{equation}
and
$$
(\beta \hat G_{\beta}u,v)=\int_{U\times
U}u(x)v(y)\hat\sigma_{\beta}(dxdy)
$$
for $u,v\in L^2(U;m)$.

(ii) There exist a unique positive Radon measure $J$ on $U\times
U$ off the diagonal $d$ and a unique positive Radon measure $K$ on
$U$ such that for $v\in C^{\infty}_{0}(U)$ and $u\in \{g\in
C^{\infty}_{0}(U): g\ {\rm is\ constant\ on\ a\ neighbourhood\ of\
{\rm supp}}[v]\}$,
\begin{equation}\label{mnb2}
{\cal E}(u,v)=\int_{U\times U\backslash
d}2(u(y)-u(x))v(y)J(dxdy)+\int_{U}u(x)v(x)K(dx).
\end{equation}
Hereafter ${\rm supp}[u]$ denotes the support of $u$. $J$ and $K$
are called the jumping and killing measures, respectively.

(iii) $(\beta/2)\sigma_{\beta}\rightarrow J$ and
$(\beta/2)\hat\sigma_{\beta}\rightarrow \hat J$ vaguely on $U\times
U\backslash d$ as $\beta\rightarrow\infty$, where $\hat
J(dxdy):=J(dydx)$.
\end{lem}

For $\delta>0$, we define
$$
U^{\delta}:=\{x\in U:\inf_{y\in \partial U}|x-y|>\delta\}.
$$
Hereafter, for $B\subset \mathbb{R}^n$, we denote by $\partial B$ its
boundary in $\mathbb{R}^n$.

Now we can state the first main result of this paper.
\begin{thm}\label{new90} Let $U$ be an open set of $\mathbb{R}^n$ and $m$ a positive Radon measure on $U$ such that ${\rm supp}[m]=U$.  Suppose that
$(A, D(A))$ is a generator on $L^2(U;m)$ such that $A$ is a Dirichlet operator and
$C^{\infty}_0(U)\subset D(A)\cap D(\hat A)$. Let $\delta>0$ be a
constant such that $U^{\delta}\not=\emptyset$. Then, we have the
decomposition:
\begin{eqnarray}\label{repre}
(-Au,v)&=&\frac{1}{2}\sum_{i,j=1}^n\int_{U}\frac{\partial
u}{\partial x_i}\frac{\partial v}{\partial x_j}\nu_{ij}(dx)
+\sum_{i=1}^n\int_{U^{\delta}}\frac{\partial
u}{\partial x_i}(x)v(x)\nu^{\delta}_i(dx)\nonumber\\
& &+\int_{U\times U\backslash d}\sum_{i=1}^{n}(y_{i}-x_{i})\left(\frac{\partial u}{\partial y_{i}}(y)v(y)-\frac{\partial u}{\partial x_{i}}(x)v(x)\right)I_{\{|x-y|\le\delta\}}(x,y)J(dxdy)\nonumber\\
& &+\int_{U\times U\backslash d}2\left(u(y)-u(x)-\sum_{i=1}^{n}(y_{i}-x_{i})\frac{\partial u}{\partial y_{i}}(y)I_{\{|x-y|\le\delta\}}(x,y)\right)v(y)J(dxdy)\nonumber\\
& &+\int_Uu(x)v(x)K(dx),\ \ \ \ \ \ \ \ \forall u,v\in
C^{\infty}_0(U^{\delta}),
\end{eqnarray}
where $J$ and $K$ are the jumping and killing measures,
respectively, $\{\nu_{ij}\}_{i,j=1}^n$ are signed Radon measures
on $U$ such that for any compact set $K\subset U$,
$\nu_{ij}(K)=\nu_{ji}(K)$ and
$\sum_{i,j=1}^n\xi_i\xi_j\nu_{ij}(K)\ge 0$ for all
$(\xi_1,\dots,\xi_n)\in \mathbb{R}^n$, and
$\{\nu^{\delta}_i\}_{i=1}^n$ are signed Radon measures on
$U^{\delta}$.
\end{thm}

On the one hand, Theorem \ref{new90} is a result of analysis, which characterizes a large class of Dirichlet generators on $\mathbb{R}^n$; one the other hand, it generalizes the classical result of Courr\`{e}ge from the Feller process setting to the right continuous Markov process setting. The representation (\ref{repre}) improves our understanding of Markov processes and has many potential applications. For example, it sheds light on the long-standing open problem, ``when does a Markov process satisfy Hunt's hypothesis (H)?" (cf. \cite{BG68,BG70,F1,GR86, HMS11,HS11,R88,Si77} and the references therein). For a dual diffusion on an open set of $\mathbb{R}^n$, (\ref{repre}) indicates the strong connection between Hunt's hypothesis (H) and the condition that the diffusion is locally associated with a semi-Dirichlet form. Here we would like to point out that Theorem \ref{new90} does not assume the sector condition although its proof is motivated by the theory of Dirichlet forms, and that the assumption $C^{\infty}_0(U)\subset D(A)\cap D(\hat A)$ is reasonable for many applications, for example, when the martingale problem of Markov processes is studied (cf. \cite[Chapter 4]{EK}).

If the diffusion part of $(X_t)_{t\geq 0}$ corresponds to a
differential operator with very singular coefficients, then it is not
suitable to assume that $C^{\infty}_{0}(U)\subset D(A)\cap
D(\hat A)$ any more. In this case, we will adopt the framework of
semi-Dirichlet forms to investigate the analytic structure of
$(X_t)_{t\geq 0}$. Suppose that $(A,D(A))$ satisfies the sector
condition, i.e., there exists a positive constant $\kappa$ such
that
\begin{equation}\label{kkkk1}
|((1-A)u,v)|\le \kappa((1-A)u,u)^{1/2}((1-A)v,v)^{1/2}, \ \
\forall u,v\in D(A).
\end{equation}
Note that $(A,D(A))$ satisfies the sector condition (\ref{kkkk1})
if and only if $(P_t)_{t>0}$ is an analytic semigroup (cf.
\cite[Corollary I.2.21]{MR92}). Denote by $({\cal E}, D({\cal
E}))$ the semi-Dirichlet form obtained by completing $D(A)$ w.r.t.
the $((1-A)u,u)^{1/2}$-norm. Assume that $C^{\infty}_{0}(U)\subset
D({\cal E})$. Then, one finds that Lemma \ref{zxcv22} also holds
for $({\cal E}, D({\cal E}))$. We make the following assumption.
\begin{assum}\label{as1} Let $O$ be a relatively compact
open set of $U$. Suppose that $\{f_n\}_{n=1}^{\infty}\subset
C^{\infty}_{0}(O)$ and $f\in C^{\infty}_{0}(O)$ satisfying $f_n$
and all of its partial derivatives converge uniformly to $f$ and
its corresponding partial derivatives as $n\rightarrow\infty$.
Then, ${\cal E}(f,g)=\lim_{n\rightarrow\infty}{\cal E}(f_n,g)$ and
${\cal E}(g,f)=\lim_{n\rightarrow\infty}{\cal E}(g, f_n)$
 for any $g\in
C^{\infty}_{0}(U)$.
\end{assum}
We will obtain the following L\'evy-Khintchine type representation of semi-Dirichlet forms, which generalizes the
classical Beurling-Deny formula of symmetric Dirichlet forms
on open sets of $\mathbb{R}^n$ (cf. \cite[Theorem 3.2.3]{Fu94}).
\begin{thm}\label{lll} Let $U$ be an open set of $\mathbb{R}^n$ and $m$ a positive Radon measure on $U$ such that ${\rm supp}[m]=U$. Suppose
that $({\cal E}, D({\cal E}))$ is a semi-Dirichlet form on
$L^2(U;m)$ such that $C^{\infty}_0(U)\subset D({\cal E})$ and
Assumption \ref{as1} holds. Let $\delta>0$ be a constant such that
$U^{\delta}\not=\emptyset$. Then, we have the decomposition:
\begin{eqnarray*}
{\mathcal{E}}(u,v)&=&\frac{1}{2}\sum_{i,j=1}^n\int_{U}\frac{\partial
u}{\partial x_i}\frac{\partial v}{\partial x_j}\nu_{ij}(dx)
+\sum_{i=1}^n\left<\Psi^{\delta}_i,\frac{\partial
u}{\partial x_i}v\right>\nonumber\\
& &+\int_{U\times U\backslash d}\sum_{i=1}^{n}(y_{i}-x_{i})\left(\frac{\partial u}{\partial y_{i}}(y)v(y)-\frac{\partial u}{\partial x_{i}}(x)v(x)\right)I_{\{|x-y|\le\delta\}}(x,y)J(dxdy)\nonumber\\
& &+\int_{U\times U\backslash d}2\left(u(y)-u(x)-\sum_{i=1}^{n}(y_{i}-x_{i})\frac{\partial u}{\partial y_{i}}(y)I_{\{|x-y|\le\delta\}}(x,y)\right)v(y)J(dxdy)\nonumber\\
& &+\int_Uu(x)v(x)K(dx),\ \ \ \ \ \ \ \ \forall u,v\in
C^{\infty}_0(U^{\delta}),
\end{eqnarray*}
where $J$ and $K$ are the jumping and killing measures,
respectively, $\{\nu_{ij}\}_{i,j=1}^n$ are signed Radon measures
on $U$ such that for any compact set $K\subset U$,
$\nu_{ij}(K)=\nu_{ji}(K)$ and
$\sum_{i,j=1}^n\xi_i\xi_j\nu_{ij}(K)\ge 0$ for all
$(\xi_1,\dots,\xi_n)\in \mathbb{R}^n$, and
$\{\Psi^{\delta}_i\}_{i=1}^n$ are generalized functions on
$U^{\delta}$.
\end{thm}

We will prove Theorems \ref{new90} and \ref{lll} in Section 2.  If
Assumption \ref{as1} is replaced by the assumption that $({\cal
E}, D({\cal E}))$ is locally controlled by Dirichlet forms, then
we can obtain a clearer characterization of the generalized
functions $\{\Psi_i^{\delta}\}_{i=1}^n$ given in Theorem \ref{lll},
see Corollary \ref{dfg} below.

In Section 3, we will apply some
ideas of Section 2 to investigate the structure of
general regular semi-Dirichlet forms. Recently, there is new
interest in further developing the theory of semi-Dirichlet forms.
For example, semi-Dirichlet forms are used to construct and study
jump-type Hunt processes (\cite{FU, SW}), the stochastic calculus
of nearly-symmetric Markov processes has been generalized to the
semi-Dirichlet form setting (\cite{MMS, O13, W}). However, the
structure of semi-Dirichlet forms is still not completely known
until now.

Let us first recall some known results on the structures of
Dirichlet forms and semi-Dirichlet forms. For notation and
terminology used in the paper, we refer to \cite{Fu94, MR92}. Suppose that $({\cal E}, D({\cal E}))$ is a regular symmetric Dirichlet
form on $L^2(E;m)$, where $E$ is a locally compact separable
metric space and $m$ is a positive Radon measure on $E$ with ${\rm
supp}[m]=E$. Recall that ``regular" implies

\noindent (i) $C_{0}(E)\cap D({\cal E})$ is dense in $D({\cal E})$ w.r.t. the $\tilde{{\cal E}}_{1}^{1/2}$-norm.

\noindent (ii) $C_{0}(E)\cap D({\cal E})$ is dense in $C_{0}(E)$ w.r.t. the uniform norm $\|\cdot\|_{\infty}$.

\noindent The Beurling-Deny formula tells us that $({\cal E},D({\cal E}))$ can be expressed for $u,v\in C_0(E)\cap D({\cal E})$ as
\begin{eqnarray}\label{BD1}
\mathcal{E}(u,v)&=&\mathcal{E}^{c}(u,v)+\int_{E\times
E\backslash
d}(u(x)-u(y))(v(x)-v(y))J(dxdy)\nonumber\\
& &+\int_{E}u(x)v(x)K(dx).
\end{eqnarray}
Here $\mathcal{E}^{c}(u,v)$ is a symmetric bilinear form with domain
$D({\cal E}^c)= C_0(E)\cap D({\cal E})$ and satisfies the strong
local property:
$$
{\cal E}^c(u,v)=0\ {\rm for}\ u\in D({\cal E}^c)\ {\rm and}\ v\in I(u),
$$
where
$$
I(u):=\{g\in D({\cal E}^c):g\ {\rm is\ constant\ on\ a\ neighbourhood\ of\ {\rm supp}}[u]\}.
$$
$J$ is a
symmetric positive Radon measure on $E\times E\backslash d$ and
$K$ is a positive Radon measure on $E$. Such ${\cal E}^c$, $J$ and
$K$ are uniquely determined by ${\cal E}$.

Furthermore, the structure of ${\cal E}^c$ is characterized by the mutual energy measures.
Let $u,v\in C_0(E)\cap D({\cal E})$. Then, there exists a unique signed Radon measure $\mu^c_{<u,v>}$ on $E$ such that
$$
\int_Efd\mu^c_{<u,v>}={\cal E}^c(uf,v)+{\cal E}^c(vf,u)-{\cal E}^c(uv,f),\ \ f\in  C_0(E)\cap D({\cal E}).
$$
We have ${\cal E}^c(u,v)=\frac{1}{2}\mu^c_{<u,v>}(E)$ and $\mu^c_{<u,v>}$ obeys LeJan's transformation rule:
$$
d\mu^c_{<\Phi(u_1,\dots,u_m),v>}=\sum_{i=1}^m\Phi_{x_i}({u_1,\dots,u_m})d\mu^c_{<u_i,v>},
$$
for any $\Phi\in C^1(\mathbb{R}^m)$ with $\Phi(0)=0$ and $u_1,\dots,u_m,v\in C_0(E)\cap D({\cal E})$.

Proofs of the above structure results on symmetric Dirichlet forms
can be found in \cite[\S 3.2]{Fu94}.  When non-symmetric Dirichlet
forms, or more generally, semi-Dirichlet forms are considered,
things become complicated. Through introducing the SPV integrable
condition, \cite{HC06} has generalized
(\ref{BD1}) to the semi-Dirichlet forms setting. Suppose that $({\cal
E}, D({\cal E}))$ is a regular semi-Dirichlet form on $L^2(E;m)$.
Then, there exist a unique positive Radon measure $J$ on $E\times
E\backslash d$ and a unique positive Radon measure $K$ on $E$ such
that for $v\in C_{0}(E)\cap D(\mathcal{E})$ and $u\in I(v)$,
\begin{eqnarray*}
\mathcal{E}(u,v)=\int_{E\times E\backslash
d}2(u(y)-u(x))v(y)J(dxdy)+\int_{E}u(x)v(x)K(dx).
\end{eqnarray*}
Define $\mathcal{A}(v):=\{f\in C_{0}(E)\cap
D(\mathcal{E}):(f(y)-f(x))v(y) \mbox{ is SPV integrable w.r.t. }
J\}$. Then, for $v\in C_{0}(E)\cap D(\mathcal{E})$ and $u\in
\mathcal{A}(v)$, we have the unique decomposition:
\begin{eqnarray}\label{BD2}
\mathcal{E}(u,v)&=&\mathcal{E}^{c}(u,v)+SPV\int_{E\times
E\backslash
d}2(u(y)-u(x))v(y)J(dxdy)\nonumber\\
& &+\int_{E}u(x)v(x)K(dx),
\end{eqnarray}
where $\mathcal{E}^{c}(u,v)$ satisfies the {\it left strong local
property} in the sense that $I(v)\subset \mathcal{A}(v)$ and
$\mathcal{E}^{c}(u,v)=0$ whenever $v\in C_0(E)\cap
D(\mathcal{E})$ and $u\in I(v)$. In general, the SPV integrable condition cannot be dropped for the decomposition (\ref{BD2}) to hold (see \cite{HMS} for an example).

\cite{HMS2,HMS} investigate the structure of non-symmetric
Dirichlet forms and characterize their diffusion parts. Suppose that
$({\cal E}, D({\cal E}))$ is a regular (non-symmetric) Dirichlet
form. Since the dual form $(\hat{\mathcal{E}},D(\mathcal{E}))$ of
$({\cal E}, D({\cal E}))$ also satisfies the semi-Dirichlet
property, we have the decomposition:
\begin{eqnarray}\label{BD3}
\hat\mathcal{E}(u,v)&=&\hat\mathcal{E}^{c}(u,v)+SPV\int_{E\times
E\backslash
d}2(u(y)-u(x))v(y)\hat J(dxdy)\nonumber\\
& &+\int_{E}u(x)v(x)\hat K(dx)
\end{eqnarray}
for $v\in C_{0}(E)\cap D(\mathcal{E})$ and $u\in
\hat\mathcal{A}(v):=\{f\in C_{0}(E)\cap
D(\mathcal{E}):(f(y)-f(x))v(y) \mbox{ is SPV integrable w.r.t. }
\hat J\}$. Note that $\hat J(dxdy)=J(dydx)$ and it can be shown that $\hat\mathcal{A}(v)=\mathcal{A}(v)$ for Dirichlet forms (cf. \cite{HMS}). Let $u,v\in C_{0}(E)\cap D(\mathcal{E})$ satisfying $(u(y)-u(x))v(y)$ is SPV integrable w.r.t. $J$. By (\ref{BD2}) and (\ref{BD3}), we get
\begin{eqnarray*}
\check{{\mathcal{E}}}(u,v)&:=&\frac{1}{2}({\mathcal{E}}(u,v)-{\mathcal{E}}(v,u))\nonumber\\
&=&\frac{1}{2}({\mathcal{E}}^c(u,v)-\hat{{\mathcal{E}}}^{c}(u,v))+SPV\int_{E\times
E\backslash d}2(u(y)-u(x))v(y)\frac{J-\hat{J}}{2}(dxdy)\nonumber\\
& &+\int_E u(x)v(x)\frac{K-\hat{K}}{2}(dx).
\end{eqnarray*}
Define $$
\check{\mathcal{E}}^c(u,v):=\frac{1}{2}({\mathcal{E}}^c(u,v)-\hat{{\mathcal{E}}}^{c}(u,v))
$$
and refer it as the \emph{co-symmetric}
diffusion part. Then, the diffusion part ${\mathcal{E}}^c$ is uniquely decomposed
into the symmetric part and the co-symmetric part as follows:
\begin{eqnarray*}
{\mathcal{E}}^c(u,v)=\tilde{{\mathcal{E}}}^c(u,v)+\check{{\mathcal{E}}}^c(u,v),
\end{eqnarray*}
where $\tilde{{\mathcal{E}}}$ denotes the symmetric part of $\mathcal{E}$ and $(\tilde{{\cal E}}, D({\cal E}))$ is a regular symmetric Dirihclet form.

Since $\tilde{\mathcal{E}}^c$ obeys LeJan's transformation rule,
to understand the structure of ${\cal E}$, we need only
concentrate on $\check{{\mathcal{E}}}^c$. In \cite{HMS},  a LeJan
type transformation rule is derived for $\check{{\mathcal{E}}}^c$
under the SPV integrable condition. This result has been used to
study Markov processes associated with non-symmetric Dirichlet
forms. For example, it plays a crucial role in investigating the
strong continuity of generalized Feynman-Kac semigroups for
nearly-symmetric Markov processes (see \cite{MS}).

In Section 3 of this paper, we will generalize the LeJan type transformation
rule of \cite{HMS} to  the semi-Dirichlet forms setting, see Theorems \ref{thm07} and \ref{mn1} below. Note that if $({\cal E},
D({\cal E}))$ is only a semi-Dirichlet form, its dual form
$(\hat{\mathcal{E}},D(\mathcal{E}))$  generally does not satisfy
the semi-Dirichlet property. So we do not have the decomposition
(\ref{BD3}). In particular, the existence of the dual killing
measure $\hat K$ is not ensured. Also, the symmetric part
$\tilde{{\mathcal{E}}}$ of $\mathcal{E}$ is only a symmetric
positivity preserving form but not a symmetric Dirichlet form,
which causes extra difficulty in characterizing the structure of
${\cal E}$.

We hope the L\'evy-Khintchine type representation and the LeJan
type transformation rule obtained in this paper can help us better
understand semi-Dirichlet forms and further their applications. We
will apply these results in a forthcoming work to consider the
strong continuity of generalized Feynman-Kac semigroups for Markov
processes associated with semi-Dirichlet forms. We refer the interested reader to \cite{AM1, CS06, C08, FK04, G94, MS} and the references therein for the topic of perturbation of Markov processes and Dirichlet forms. Finally, we would
like to point out that by quasi-homeomorphisms (cf.
\cite{CMR,HC06,Kuwae}) many results obtained in this paper can be
extended to quasi-regular semi-Dirichlet forms.

\section [short title]{L\'evy-Khintchine type
representation of Dirichlet generators and semi-Dirichlet forms on
open sets of $\mathbb{R}^n$} \setcounter{equation}{0}

Throughout this section, we let $U$ be an open set of
$\mathbb{R}^n$ which is equipped with the subspace topology of
$\mathbb{R}^n$ and $m$ a positive Radon measure on $U$ such that
${\rm supp}[m]=U$. We will give a L\'evy-Khintchine type
representation for Dirichlet generators and semi-Dirichlet forms on $U$. All the results of this section, except for those given in \S 2.3, apply to both of the following two cases.

\noindent Case 1: $(A,D(A))$ is a Dirichlet operator on
$L^2(U;m)$ and is the generator of a strongly continuous
contraction semigroup on $L^2(U;m)$. We assume that
$C^{\infty}_{0}(U)\subset D(A)\cap D(\hat A)$ and define the bilinear form ${\cal E}$ as in
(\ref{cvbn}).

\noindent Case 2: $({\cal E}, D({\cal E}))$ is a semi-Dirichlet
form on $L^2(U;m)$ such that $C^{\infty}_0(U)\subset D({\cal E})$
and Assumption \ref{as1} holds.

Let $J$ be the jumping measure given in Lemma \ref{zxcv22}. We
choose a sequence of relatively compact open sets
$\Omega_l\uparrow U$ and a sequence of numbers
$\varsigma_l\downarrow 0$ such that the set $\Gamma_l:=\{(x,y)\in
\Omega_l\times \Omega_l:|x-y|\ge \varsigma_l\}$ is a continuous
set w.r.t. $J$ for every $l\in \mathbb{N}$. Hereafter when we say
that a set $B$ is a relatively compact set of an open set $V$ of
$\mathbb{R}^n$, we mean that $B\subset V$ and $B$ is relatively
compact w.r.t. the subspace topology of $V$ inherited from
$\mathbb{R}^n$. Denote $\Lambda_l:=\{(x,y)\in \Omega_l\times
\Omega_l:|x-y|<\varsigma_l\}$. Define $\hat{\cal E}(u,v):={\cal E}(v,u)$ for $u,v\in C^{\infty}_{0}(U)$.

\subsection{Decomposition of ${\cal E}$}

\begin{lem}\label{l1} Let $u,v\in C^{\infty}_{0}(U)$ and $F$ be a compact set of $U$. Then

\noindent (i) $$ \int_{U\times F\backslash
d}(u(y)-u(x))^{2}J(dxdy)<\infty.$$

\noindent (ii)
$$\int_{F\times F\setminus d}|x-y|^{2}J(dxdy)<\infty.$$

\noindent (iii) For $\varepsilon>0$,
$$
\int_{(U\times U)\cap\{
|x-y|>\varepsilon\}}|(u(y)-u(x))v(y)|J(dxdy)<\infty.
$$
\end{lem}

 \begin{proof}\ \ (i) We choose a $w\in C^{\infty}_{0}(U)$ satisfying $w\ge 0$ and $w|_{F}\equiv 1$. By (\ref{E1}) and the sub-Markovian property of $(G_{\beta})_{\beta>0}$, we get
\begin{eqnarray*}
\int_{U\times F\setminus d}(u(y)-u(x))^{2}J(dxdy)&\le&\int_{U\times U\setminus d}(u(y)-u(x))^{2}w(y)J(dxdy)\nonumber\\
&=&\lim_{l\rightarrow\infty}\int_{\Gamma_l}(u(y)-u(x))^{2}w(y)J(dxdy)\nonumber\\
&=&\lim_{l\rightarrow\infty}\lim_{\beta\rightarrow\infty}\frac{\beta}{2}\int_{\Gamma_l}(u(y)-u(x))^{2}w(y)\sigma_{\beta}(dxdy)\nonumber\\
&\le&\lim_{\beta\rightarrow\infty}\frac{\beta}{2}\int_{U\times U}(u(y)-u(x))^{2}w(y)\sigma_{\beta}(dxdy)\nonumber\\
&=&\lim_{\beta\rightarrow\infty}\frac{\beta}{2}\{(\beta G_{\beta}I_U,u^2w)-2(\beta G_{\beta}u,uw)+(\beta G_{\beta}u^2,w)\}\nonumber\\
&\le&\lim_{\beta\rightarrow\infty}\left\{\beta(u-\beta G_{\beta}u,uw)-\frac{\beta}{2}(u^2-\beta G_{\beta}u^2,w)\right\}\nonumber\\
&=&{\cal E}(u,uw)-\frac{1}{2}{\cal E}(u^2,w)\nonumber\\
&<&\infty.
\end{eqnarray*}

(ii) We choose a $w'\in C^{\infty}_{0}(U)$ satisfying
$w'|_{F}\equiv 1$. For $1\le i \le n$, we define $u_i(x)=x_i\cdot
w'(x)$ for $x=(x_1,\dots,x_n)\in U$. Then, $u_i\in
C^{\infty}_0(U)$ satisfying $u_i(x)=x_i$ for $x\in F$. By (i), we
get
\begin{eqnarray*}
\int_{F\times F\setminus d}|x-y|^{2}J(dxdy)&=&\sum_{i=1}^{n}\int_{F\times F\setminus d}(x_{i}-y_{i})^{2}J(dxdy)\\
&=&\sum_{i=1}^{n}\int_{F\times F\setminus d}(u_{i}(x)-u_{i}(y))^{2}J(dxdy)\\
&<&\infty.
\end{eqnarray*}

(iii) By (i), we get
\begin{eqnarray*}
& &\int_{(U\times U)\cap\{ |x-y|>\varepsilon\}}|(u(y)-u(x))v(y)|J(dxdy)\\
&=&\int_{(U\times {\rm supp}[v])\cap\{|x-y|>\varepsilon\}}|(u(y)-u(x))v(y)|J(dxdy)\\
&=&\int_{(U\times {\rm supp}[v])\cap\{|x-y|>\varepsilon\}}|(u(y)-u(x))(v(y)-v(x))+(u(y)-u(x))v(x)|J(dxdy)\\
&\leq&\int_{U\times {\rm supp}[v]\setminus d}|(u(y)-u(x))(v(y)-v(x))|J(dxdy)\\
& &+\int_{({\rm supp}[v]\times {\rm supp}[v])\cap\{|x-y|>\varepsilon\}}|(u(y)-u(x))v(x)|J(dxdy)\\
&\leq&\left(\int_{U\times {\rm supp}[v]\setminus d}(u(y)-u(x))^{2}J(dxdy)\right)^{1/2}\left(\int_{U\times {\rm supp}[v]\setminus d}(v(y)-v(x))^{2}J(dxdy)\right)^{1/2}\\
&&+2\|u\|_{\infty}\|v\|_{\infty}J(({\rm supp}[v]\times {\rm supp}[v])\cap \{|x-y|>\varepsilon\})\\
&<&\infty.
\end{eqnarray*}
\end{proof}

Let $\delta>0$ be a constant such that $U^{\delta}\not=\emptyset$.
Suppose that $u,v\in C^{\infty}_{0}(U^{\delta})$. Let $\chi\in
C^{\infty}_{0}(U)$ satisfying $\chi=1$ on a neighbourhood of ${\rm
supp}[u]\cup{\rm supp}[v]$.  By Taylor's theorem and Lemma
\ref{l1}(ii), one finds that
$(u(y)-u(x)-\sum_{i=1}^{n}(y_{i}-x_{i})\frac{\partial u}{\partial
y_{i}}(y)I_{\{|x-y|\le\delta\}}(x,y))v(y)\chi(x)$ is integrable
w.r.t. both $J$ and $\hat J$. Hereafter, we define
\begin{equation}\label{eqa}
F^{\delta}_v:=\left\{x\in U:\inf_{y\in {\rm supp}[v]}|x-y|\le\delta\right\}.
\end{equation}
$F^{\delta}_v$ is a compact set of $U$. By Lemma \ref{l1}(i) and (ii), for $1\le i\le n$, we have
\begin{eqnarray*}
& &\int_{U\times U\backslash d}\left|(y_{i}-x_{i})\frac{\partial u}{\partial y_{i}}(y)I_{\{|x-y|\le\delta\}}(x,y)v(y)(1-\chi(x))\right|(J(dxdy)+\hat J(dxdy))\nonumber\\
&\le&2\left\|\frac{\partial u}{\partial y_{i}}\cdot v\right\|_{\infty}\left(\int_{F^{\delta}_v\times F^{\delta}_v\backslash d}|x-y|^2J(dxdy)\right)^{1/2}\left(\int_{F^{\delta}_v\times F^{\delta}_v\backslash d}(\chi(y)-\chi(x))^2J(dxdy)\right)^{1/2}\nonumber\\
&<&\infty.
\end{eqnarray*}
Hence $\sum_{i=1}^n(y_{i}-x_{i})\frac{\partial u}{\partial
y_{i}}(y)I_{\{|x-y|\le\delta\}}(x,y)v(y)(1-\chi(x))$ is integrable
w.r.t. both $J$ and $\hat J$. Therefore,
$$\left((u(y)-u(x))\chi(x)-\sum_{i=1}^{n}(y_{i}-x_{i})\frac{\partial
u}{\partial y_{i}}(y)I_{\{|x-y|\le\delta\}}(x,y)\right)v(y)$$ is
integrable w.r.t. both $J$ and $\hat J$.

We assume temporarily that $J(\{(x,y)\in U\times
U:|x-y|=\delta\})=0$. Then, we obtain by the vague convergence of
$(\beta/2)\sigma_{\beta}$ to $J$ that
\begin{eqnarray}\label{add2}
& &{\mathcal{E}}(u,v)=\lim_{\beta\rightarrow\infty}\beta(u-\beta {G}_{\beta}u,v)\nonumber\\
&=&\lim_{\beta\rightarrow\infty}\beta\left\{\int_{U\times U}(u(y)-u(x))v(y)\chi(x){\sigma}_{\beta}(dxdy)\right.\nonumber\\
& &\ \ \ \ +\left.\int_{U}\chi(x)u(x)v(x)m(dx)-\int_{U\times U}\chi(x)u(y)v(y){\sigma}_{\beta}(dxdy)\right\}\nonumber\\
&=&\lim_{\beta\rightarrow\infty}\beta\int_{U\times U}(u(y)-u(x))v(y)\chi(x){\sigma}_{\beta}(dxdy)+{\cal E}(\chi,uv)\nonumber\\
&=&\lim_{l\rightarrow\infty}\lim_{\beta\rightarrow\infty}\beta\left\{\int_{\Lambda_l}(u(y)-u(x))v(y)\sigma_{\beta}(dxdy)\right.\nonumber\\
& &\ \ \ \ \left.+\int_{\Gamma_l}\left((u(y)-u(x))\chi(x)-\sum_{i=1}^{n}(y_{i}-x_{i})\frac{\partial u}{\partial y_{i}}(y)I_{\{|x-y|\le\delta\}}(x,y)\right)v(y)\sigma_{\beta}(dxdy)\right.\nonumber\\
&&\ \ \ \
+\left.\int_{\Gamma_l}\sum_{i=1}^{n}(y_{i}-x_{i})\frac{\partial
u}{\partial
y_{i}}(y)I_{\{|x-y|\le\delta\}}(x,y)v(y){\sigma}_{\beta}(dxdy)\right\}
+{\cal E}(\chi,uv)\nonumber\\
&=&\lim_{l\rightarrow\infty}\lim_{\beta\rightarrow\infty}\beta\left\{\int_{\Lambda_l}(u(y)-u(x))v(y)\sigma_{\beta}(dxdy)\right.\nonumber\\
& & \ \ \ \ +\left.\int_{\Gamma_l}\sum_{i=1}^{n}(y_{i}-x_{i})\frac{\partial u}{\partial y_{i}}(y)I_{\{|x-y|\le\delta\}}(x,y)v(y){\sigma}_{\beta}(dxdy)\right\}\nonumber\\
& &+\int_{U\times U\backslash d}2\left((u(y)-u(x))\chi(x)-\sum_{i=1}^{n}(y_{i}-x_{i})\frac{\partial u}{\partial y_{i}}(y)I_{\{|x-y|\le\delta\}}(x,y)\right)v(y)J(dxdy)\nonumber\\
& &+{\cal E}(\chi,uv).
\end{eqnarray}
Similarly, we get
\begin{eqnarray}\label{add3}
& &\hat{\mathcal{E}}(u,v)=\lim_{l\rightarrow\infty}\lim_{\beta\rightarrow\infty}\beta
\left\{\int_{\Lambda_l}(u(y)-u(x))v(y)\hat\sigma_{\beta}(dxdy)\right.\nonumber\\
& &\ \ \ \ \ \ \ \ \ \ \ \ \ \ \ \ \ +\left.\int_{\Gamma_l}\sum_{i=1}^{n}(y_{i}-x_{i})\frac{\partial u}{\partial y_{i}}(y)I_{\{|x-y|\le\delta\}}(x,y)v(y){\hat\sigma}_{\beta}(dxdy)\right\}\nonumber\\
& &+\int_{U\times U\backslash d}2\left((u(y)-u(x))\chi(x)-\sum_{i=1}^{n}(y_{i}-x_{i})\frac{\partial u}{\partial y_{i}}(y)I_{\{|x-y|\le\delta\}}(x,y)\right)v(y)\hat J(dxdy)\nonumber\\
& &+\hat{\cal E}(\chi,uv).
\end{eqnarray}

By (\ref{add2}) and (\ref{add3}), we can introduce the following
definition.
\begin{defin}\label{def1} Let $\{\delta_n\}_{n=1}^{\infty}$ be a sequence of constants satisfying $\delta=\lim_{n\rightarrow\infty}\delta_n$, $\delta_n\ge \delta$ and $J(\{(x,y)\in U\times
U:|x-y|=\delta_n\})=0$ for each $n\in \mathbb{N}$. For $u,v\in
C^{\infty}_{0}(U^{\delta})$, we define
\begin{eqnarray}\label{dd11}
& &{\mathcal{E}}^{c,\delta}(u,v)
:=\lim_{n\rightarrow\infty}\lim_{l\rightarrow\infty}\lim_{\beta\rightarrow\infty}\beta\left\{\int_{\Lambda_l}(u(y)-u(x))v(y)\sigma_{\beta}(dxdy)\right.\nonumber\\
& & \ \ \ \ \ \ \ \  \ \ \ \ \ \ \ \
+\left.\int_{\Gamma_l}\sum_{i=1}^{n}(y_{i}-x_{i})\frac{\partial
u}{\partial
y_{i}}(y)I_{\{|x-y|\le\delta_n\}}(x,y)v(y){\sigma}_{\beta}(dxdy)\right\}\
\ \ \
\end{eqnarray}
and
\begin{eqnarray}\label{dd22}
&
&\hat{\mathcal{E}}^{c,\delta}(u,v):=\lim_{n\rightarrow\infty}\lim_{l\rightarrow\infty}\lim_{\beta\rightarrow\infty}\beta\left\{\int_{\Lambda_l}
(u(y)-u(x))v(y)\hat\sigma_{\beta}(dxdy)\right.\nonumber\\
& & \ \ \ \ \ \ \ \  \ \ \ \ \ \ \ \
+\left.\int_{\Gamma_l}\sum_{i=1}^{n}(y_{i}-x_{i})\frac{\partial
u}{\partial
y_{i}}(y)I_{\{|x-y|\le\delta_n\}}(x,y)v(y){\hat\sigma}_{\beta}(dxdy)\right\}.\
\ \ \
\end{eqnarray}
\end{defin}

By (\ref{add2}), (\ref{add3}) and the fact that $J$ is a positive
Radon measure $J$ on $U\times U\backslash d$, one finds that the
definitions of ${\mathcal{E}}^{c,\delta}$ and
$\hat{\mathcal{E}}^{c,\delta}$ are independent of the selections
of $\{\Omega_l\}$ and $\{\delta_n\}$. Both
$\mathcal{E}^{c,\delta}(u,v)$ and
$\hat\mathcal{E}^{c,\delta}(u,v)$ satisfy the left strong local
property in the sense that
${\mathcal{E}}^{c,\delta}(u,v)=\hat{\mathcal{E}}^{c,\delta}(u,v)=0$
whenever $u$ is constant on a neighbourhood of  ${\rm supp}[v]$.

\begin{thm} Suppose $u,v\in C^{\infty}_{0}(U^{\delta})$.

\noindent (i) We have the decomposition
\begin{eqnarray}\label{E102}
& &\mathcal{E}(u,v)={\mathcal{E}}^{c,\delta}(u,v)\nonumber\\
& &\ \ \ \ +\int_{U\times U\backslash d}2\left(u(y)-u(x)-\sum_{i=1}^{n}(y_{i}-x_{i})\frac{\partial u}{\partial y_{i}}(y)I_{\{|x-y|\le\delta\}}(x,y)\right)v(y)J(dxdy)\nonumber\\
& &\ \ \ \ +\int_Uu(x)v(x)K(dx).
\end{eqnarray}

\noindent (ii) Let $\chi\in C^{\infty}_{0}(U)$ satisfying $\chi=1$ on a neighbourhood of  ${\rm supp}[u]\cup{\rm supp}[v]$. Then, we have
\begin{eqnarray}\label{znew}
& &\mathcal{E}(u,v)={\mathcal{E}}^{c,\delta}(u,v)\nonumber\\
& &\ \ \ \ +\int_{U\times U\backslash d}2\left((u(y)-u(x))\chi(x)-\sum_{i=1}^{n}(y_{i}-x_{i})\frac{\partial u}{\partial y_{i}}(y)I_{\{|x-y|\le\delta\}}(x,y)\right)v(y)J(dxdy)\nonumber\\
& &\ \ \ \ +{\cal E}(\chi,uv)
\end{eqnarray}
and
\begin{eqnarray}\label{z1}
& &\hat\mathcal{E}(u,v)=\hat{\mathcal{E}}^{c,\delta}(u,v)\nonumber\\
& &\ \ \ \ +\int_{U\times U\backslash d}2\left((u(y)-u(x))\chi(x)-\sum_{i=1}^{n}(y_{i}-x_{i})\frac{\partial u}{\partial y_{i}}(y)I_{\{|x-y|\le\delta\}}(x,y)\right)v(y)\hat J(dxdy)\nonumber\\
& &\ \ \ \ +\hat {\cal E}(\chi,uv).
\end{eqnarray}
\end{thm}

\begin{proof}\ \ (ii) is a direct consequence of (\ref{add2})--(\ref{dd22}). We only prove
(i). By (\ref{mnb2}), we have
\begin{equation}\label{E111}
\mathcal{E}(\chi, uv)=\int_{U\times U\backslash
d}2(1-\chi(x))u(y)v(y)J(dxdy)+\int_{U}u(x)v(x)K(dx).
\end{equation}
Here the integrability of $(1-\chi(x))u(y)v(y)$ w.r.t. $J$ is also
ensured by Lemma \ref{l1}(iii).  Then, we obtain (\ref{E102}) by
(\ref{znew}) and (\ref{E111}).
\end{proof}

By (\ref{E102}), to understand the structure of ${\cal E}$, we may
concentrate on the left strong local part ${\cal E}^{c,\delta}$.

Suppose that $u,f\in C_{0}^{\infty}(U^{\delta})$. By (\ref{dd11}),
we get
\begin{equation}\label{kl13}
2{\mathcal{E}}^{c,\delta}(u,uf)-{\mathcal{E}}^{c,\delta}(u^{2},f)=\lim_{l\rightarrow\infty}\lim_{\beta\rightarrow\infty}\beta\int_{\Lambda_l}(u(y)-u(x))^2f(y)\sigma_{\beta}(dxdy).
\end{equation}
Since $\delta$ is arbitrary,
$\lim_{l\rightarrow\infty}\lim_{\beta\rightarrow\infty}\beta\int_{\Lambda_l}(\varphi(y)-\varphi(x))^2g(y)\sigma_{\beta}(dxdy)$
exists for any $\varphi,g\in C_{0}^{\infty}(U)$.

 Let $\varphi\in C_{0}^{\infty}(U)$. For
$r\in\mathbb{N}$, we choose a $w\in C^{\infty}_{0}(U)$ satisfying
$w\ge 0$ and $w|_{\Omega_r}\equiv 1$. For $g\in
C_{0}^{\infty}(\Omega_r)$, we obtain by the sub-Markovian property
of $(G_{\beta})_{\beta>0}$ that
\begin{eqnarray*}
&&\left|\lim_{l\rightarrow\infty}\lim_{\beta\rightarrow\infty}\beta\int_{\Lambda_l}(\varphi(y)-\varphi(x))^2g(y)\sigma_{\beta}(dxdy)\right|\\
&\le&\|g\|_{\infty}\lim_{\beta\rightarrow\infty}\beta\int_{U\times U}(\varphi(y)-\varphi(x))^2w(y)\sigma_{\beta}(dxdy)\\
&\le&\|g\|_{\infty}\lim_{\beta\rightarrow\infty}\{2\beta(\varphi-\beta G_{\beta}\varphi,\varphi w)-\beta(\varphi^2-\beta G_{\beta}\varphi^2,w)\}\\
&=&(2{\cal E}(\varphi,\varphi w)-{\cal
E}(\varphi^2,w))\|g\|_{\infty}.
\end{eqnarray*}
Then, there exists a unique Radon measure $\mu_{<\varphi>}^{r,c}$
on $\Omega_r$ such that
$$
\int_{\Omega_r}gd\mu_{<\varphi>}^{r,c}=\lim_{l\rightarrow\infty}\lim_{\beta\rightarrow\infty}\beta\int_{\Lambda_l}(\varphi(y)-\varphi(x))^2g(y)\sigma_{\beta}(dxdy),\
\ \forall g\in C_{0}^{\infty}(\Omega_r).
$$
It is easy to see that $\{\mu_{<\varphi>}^{r,c}\}$ is a consistent
sequence of Radon measures. Therefore, we can well define the
measure $\mu^c_{<\varphi>}$ by
$\mu^c_{<\varphi>}=\mu^{r,c}_{<\varphi>}$ on $\Omega_r$, which
satisfies
$$
\int_{U}gd\mu_{<\varphi>}^{c}=\lim_{l\rightarrow\infty}\lim_{\beta\rightarrow\infty}\beta\int_{\Lambda_l}(\varphi(y)-\varphi(x))^2g(y)\sigma_{\beta}(dxdy),\
\ \forall g\in C_{0}^{\infty}(U).
$$

For $\varphi,\phi\in C_{0}^{\infty}(U)$, we define
$$
\mu_{<\varphi,\phi>}^c:=\frac{1}{2}(\mu_{<\varphi+\phi>}^c-\mu_{<\varphi>}^c-\mu_{<\phi>}^c).
$$
Then, for any $g\in C_{0}^{\infty}(U)$, we have
\begin{equation}\label{hgfv}
\int_{U}gd\mu_{<\varphi,\phi>}^c=\lim_{l\rightarrow\infty}\lim_{\beta\rightarrow\infty}\beta\int_{\Lambda_l}(\varphi(y)-\varphi(x))(\phi(y)-\phi(x))g(y)
\sigma_{\beta}(dxdy).
\end{equation}

Suppose now that $u,v,f\in C_{0}^{\infty}(U^{\delta})$. We obtain
by (\ref{kl13}) and (\ref{hgfv}) that
\begin{equation}\label{ee1}
\int_{U}fd\mu_{<u,v>}^c={\mathcal{E}}^{c,\delta}(u,vf)+{\mathcal{E}}^{c,\delta}(v,uf)-{\mathcal{E}}^{c,\delta}(uv,f).
\end{equation}
Hence, for any $h\in C_{0}^{\infty}(U^{\delta})$ satisfying $h|_{{\rm supp}[u]\cup{\rm supp}[v]}\equiv1$, we have
\begin{equation}\label{eee0}
{\mathcal{E}}^{c,\delta}(u,v)+{\mathcal{E}}^{c,\delta}(v,u)=\int_{U}hd\mu_{<u,v>}^c+{\mathcal{E}}^{c,\delta}(uv,h).
\end{equation}

For $u,v\in C_{0}^{\infty}(U^{\delta})$, we define a linear
functional $L^{\delta}(u,v)$ on $C_{0}^{\infty}(U^{\delta})$ by
\begin{equation}\label{e01}<L^{\delta}(u,v),f>:=\frac{1}{2}({\mathcal{E}}^{c,\delta}(u,vf)-\hat{\mathcal{E}}^{c,\delta}(u,vf)),\ \ f\in C_{0}^{\infty}(U^{\delta}).
\end{equation}
Then, for any $h\in C_{0}^{\infty}(U^{\delta})$ satisfying
$h|_{{\rm supp}[v]}\equiv1$, we have
\begin{equation}\label{eee1}
{\mathcal{E}}^{c,\delta}(u,v)-\hat{\mathcal{E}}^{c,\delta}(u,v)=2<L^{\delta}(u,v),h>.
\end{equation}

Let $\chi\in C^{\infty}_{0}(U^{\delta})$ satisfying $\chi=1$ on a
neighbourhood of  ${\rm supp}[u]\cup{\rm supp}[v]$. Then, we
obtain by (\ref{E102}), (\ref{z1}) and (\ref{eee1}) that
\begin{eqnarray}\label{eee44}
& &{\mathcal{E}}^{c,\delta}(u,v)-{\mathcal{E}}^{c,\delta}(v,u)\nonumber\\
&=&{\mathcal{E}}^{c,\delta}(u,v)-\mathcal{E}(v,u)+\int_Uu(x)v(x)K(dx)\nonumber\\
&&+\int_{U\times U\backslash d}2\left(v(y)-v(x)-\sum_{i=1}^{n}(y_{i}-x_{i})\frac{\partial v}{\partial y_{i}}(y)I_{\{|x-y|\le\delta\}}(x,y)\right)u(y)J(dxdy)\nonumber\\
&=&{\mathcal{E}}^{c,\delta}(u,v)-\hat{\mathcal{E}}^{c,\delta}(u,v)-\hat {\cal E}(\chi,uv)+\int_Uu(x)v(x)K(dx)\nonumber\\
&&-\int_{U\times U\backslash d}2\left((u(y)-u(x))\chi(x)-\sum_{i=1}^{n}(y_{i}-x_{i})\frac{\partial u}{\partial y_{i}}(y)I_{\{|x-y|\le\delta\}}(x,y)\right)v(y)\hat J(dxdy)\nonumber\\
&&+\int_{U\times U\backslash d}2\left(v(y)-v(x)-\sum_{i=1}^{n}(y_{i}-x_{i})\frac{\partial v}{\partial y_{i}}(y)I_{\{|x-y|\le\delta\}}(x,y)\right)u(y)J(dxdy)\nonumber\\
&=&2<L^{\delta}(u,v),\chi>-{\cal E}(uv,\chi)+\int_Uu(x)v(x)K(dx)\nonumber\\
&&-\int_{U\times U\backslash d}2\left((u(y)-u(x))\chi(x)-\sum_{i=1}^{n}(y_{i}-x_{i})\frac{\partial u}{\partial y_{i}}(y)I_{\{|x-y|\le\delta\}}(x,y)\right)v(y)\hat J(dxdy)\nonumber\\
&&+\int_{U\times U\backslash d}2\left(v(y)-v(x)-\sum_{i=1}^{n}(y_{i}-x_{i})\frac{\partial v}{\partial y_{i}}(y)I_{\{|x-y|\le\delta\}}(x,y)\right)u(y)J(dxdy)\nonumber\\
&=&2<L^{\delta}(u,v),\chi>-{\mathcal{E}}^{c,\delta}(uv,\chi)\nonumber\\
&&-\int_{U\times U\backslash d}2\left((uv)(y)-(uv)(x)-\sum_{i=1}^{n}(y_{i}-x_{i})\frac{\partial (uv)}{\partial y_{i}}(y)I_{\{|x-y|\le\delta\}}(x,y)\right)\chi(y)J(dxdy)\nonumber\\
&&-\int_{U\times U\backslash d}2\left((u(x)-u(y))\chi(y)-\sum_{i=1}^{n}(x_{i}-y_{i})\frac{\partial u}{\partial x_{i}}(x)I_{\{|x-y|\le\delta\}}(x,y)\right)v(x)J(dxdy)\nonumber\\
&&+\int_{U\times U\backslash d}2\left(v(y)-v(x)-\sum_{i=1}^{n}(y_{i}-x_{i})\frac{\partial v}{\partial y_{i}}(y)I_{\{|x-y|\le\delta\}}(x,y)\right)u(y)J(dxdy)\nonumber\\
&=&2<L^{\delta}(u,v),\chi>-{\mathcal{E}}^{c,\delta}(uv,\chi)\nonumber\\
&&+\int_{U\times U\backslash
d}2\sum_{i=1}^{n}(y_{i}-x_{i})\left(\frac{\partial u}{\partial
y_{i}}(y)v(y)-\frac{\partial u}{\partial
x_{i}}(x)v(x)\right)I_{\{|x-y|\le\delta\}}(x,y)J(dxdy).
\end{eqnarray}

By (\ref{eee0}) and (\ref{eee44}), we obtain the following
theorem.

\begin{thm}\label{thm07}
Suppose $u,v\in C^{\infty}_{0}(U^{\delta})$ and $\chi\in
C^{\infty}_{0}(U^{\delta})$ satisfying $\chi=1$ on a neighbourhood
of ${\rm supp}[u]\cup{\rm supp}[v]$. Then
\begin{eqnarray}\label{bb1}
& &{\mathcal{E}}^{c,\delta}(u,v)=\frac{1}{2}\int_{U}\chi d\mu_{<u,v>}^c+<L^{\delta}(u,v),\chi>\nonumber\\
& &\ \ +\int_{U\times U\backslash
d}\sum_{i=1}^{n}(y_{i}-x_{i})\left(\frac{\partial u}{\partial
y_{i}}(y)v(y)-\frac{\partial u}{\partial
x_{i}}(x)v(x)\right)I_{\{|x-y|\le\delta\}}(x,y)J(dxdy).\ \ \ \ \ \
\ \ \ \ \
\end{eqnarray}
\end{thm}

\subsection{Transformation rules for the symmetric and co-symmetric diffusion parts}

In this subsection, we will derive transformation rules for the
sign Radon measure $\mu^c_{<\cdot,\cdot>}$ and the lineal
functional $L^{\delta}(\cdot,\cdot)$ introduced in \S 2.1.

\begin{thm}\label{thm9}
(i) For $u, v, w\in C_{0}^{\infty}(U)$,
$$
d\mu^{c}_{<uv,w>}={u}d\mu^{c}_{<v,w>}+{v}d\mu^{c}_{<u,w>}.
$$

\noindent (ii) For $u, v, w, f\in
C_{0}^{\infty}(U^{\delta})$,
$$
<L^{\delta}(u,vw),f>=<L^{\delta}(u,v),wf>.
$$

\noindent (iii) For $u, v, w, f\in C_{0}^{\infty}(U^{\delta})$,
$$<L^{\delta}(uv,w),f>=<L^{\delta}(u,w),vf>+<L^{\delta}(v,w),uf>.$$
\end{thm}
\begin{proof}\ \ We assume without loss of generality that $u, v, w, f\in
C_{0}^{\infty}(U^{\delta})$.

(i) We need only show that
$d\mu^{c}_{<u^2,v>}=2{u}d\mu^{c}_{<u,v>}$. We choose a $\chi\in
C^{\infty}_{0}(U)$ satisfying $\chi=1$ on a neighbourhood of
${{\rm supp}[u]\cup {\rm supp}[v]}$. Let $g\in
C_{0}^{\infty}(U^{\delta})$. Then, by (\ref{hgfv}) and Lemma
\ref{l1}(i), we get
\begin{eqnarray*}
& &\int_{U^{\delta}}gd\mu_{<u^2,v>}^c-2\int_{U^{\delta}}gud\mu_{<u,v>}^c\\
&=&-\lim_{l\rightarrow\infty}\lim_{\beta\rightarrow\infty}\beta\int_{\Lambda_l}(u(y)-u(x))^2(v(y)-v(x))g(y)
\sigma_{\beta}(dxdy)\\
&=&-\lim_{l\rightarrow\infty}\lim_{\beta\rightarrow\infty}\beta\int_{\Lambda_l}(u(y)-u(x))^2(v(y)-v(x))g(y)\chi(x)
\sigma_{\beta}(dxdy)\\
&
&-\lim_{l\rightarrow\infty}\lim_{\beta\rightarrow\infty}\beta\int_{\Lambda_l}u^2(y)v(y)g(y)(1-\chi(x))
\sigma_{\beta}(dxdy)\\
&=&-\lim_{l\rightarrow\infty}\int_{\Lambda_l}2(u(y)-u(x))^2(v(y)-v(x))g(y)\chi(x)
J(dxdy)\\
&=&0.
\end{eqnarray*}

(ii) is obvious by (\ref{e01}).

(iii) We need only show that
$<L^{\delta}(u^2,v),f>=2<L^{\delta}(u,v),uf>$. By (\ref{dd11}),
(\ref{dd22}) and (\ref{e01}), we get
\begin{eqnarray*}
& &<L^{\delta}(u,v),f>\\
&=&\lim_{n\rightarrow\infty}\lim_{l\rightarrow\infty}\lim_{\beta\rightarrow\infty}\frac{\beta}{2}\left\{\int_{\Lambda_l}(u(y)-u(x))(v(y)f(y)+v(x)f(x))\sigma_{\beta}(dxdy)\right.\nonumber\\
& & \ \ \ \ \ \ \ \  \ \ \ \ \ \ \ \ +\left.\int_{\Gamma_l}\sum_{i=1}^{n}(y_{i}-x_{i})\frac{\partial u}{\partial y_{i}}(y)I_{\{|x-y|\le\delta_n\}}(x,y)v(y)f(y){\sigma}_{\beta}(dxdy)\right.\\
& & \ \ \ \ \ \ \ \  \ \ \ \ \ \ \ \
-\left.\int_{\Gamma_l}\sum_{i=1}^{n}(y_{i}-x_{i})\frac{\partial
u}{\partial
y_{i}}(y)I_{\{|x-y|\le\delta_n\}}(x,y)v(y)f(y){\hat\sigma}_{\beta}(dxdy)\right\}.
\end{eqnarray*}
We choose a $\chi\in C^{\infty}_{0}(U)$ satisfying $\chi=1$ on a neighbourhood of ${{\rm supp}[u]\cup {\rm supp}[v]}$. Then, by Lemma \ref{l1}(i), we get
\begin{eqnarray*}
& &<L^{\delta}(u^2,v),f>-2<L^{\delta}(u,v),uf>\\
&=&-\lim_{l\rightarrow\infty}\lim_{\beta\rightarrow\infty}\frac{\beta}{2}\int_{\Lambda_l}(u(y)-u(x))^2(v(y)f(y)-v(x)f(x))\sigma_{\beta}(dxdy)\\
&=&-\lim_{l\rightarrow\infty}\lim_{\beta\rightarrow\infty}\frac{\beta}{2}\int_{\Lambda_l}
(u(y)-u(x))^2(v(y)f(y)-v(x)f(x))\chi(x)\chi(y)\sigma_{\beta}(dxdy)\\
&&-\lim_{l\rightarrow\infty}\lim_{\beta\rightarrow\infty}\frac{\beta}{2}\int_{\Lambda_l}
u^2(y)v(y)f(y)(1-\chi(x))\sigma_{\beta}(dxdy)\\
&&+\lim_{l\rightarrow\infty}\lim_{\beta\rightarrow\infty}\frac{\beta}{2}\int_{\Lambda_l}
u^2(x)v(x)f(x)(1-\chi(y))\sigma_{\beta}(dxdy)\\
&=&-\lim_{l\rightarrow\infty}\int_{\Lambda_l}
(u(y)-u(x))^2(v(y)f(y)-v(x)f(x))\chi(x)\chi(y)J(dxdy)\\
&=&0.
\end{eqnarray*}
\end{proof}

Let $w\in C^{\infty}_0(U)$ and $V$ be a relatively compact open
set of $U$. If $w=k$ (constant) on $V$, then $\mu_{<w>}^c=0$ on
$V$.  In fact, taking an $f\in C^{\infty}_0(V)$, we obtain by
Theorem \ref{thm9}(i) that
$$
kd\mu^{c}_{<f,w>}=d\mu^{c}_{<fw,w>}=fd\mu^{c}_{<w>}+wd\mu^{c}_{<f,w>},
$$
which implies that $fd\mu^{c}_{<w>}=0$ on $V$. Since $f\in
C^{\infty}_0(V)$ is arbitrary, $\mu_{<w>}^c=0$ on $V$. For $u,v\in
C^{\infty}(U)$,  we choose a sequence of functions
$\{u_l,v_l\}\subset C^{\infty}_0(U)$ such that $u=u_l$ and $v=v_l$
on $\Omega_l$. Therefore, we can well define the measure
$\mu_{<u,v>}^c$ by $\mu_{<u,v>}^c=\mu_{<u_l,v_l>}^c$ on
$\Omega_l$. The definition of $\mu_{<u,v>}^c$ is  independent of
the selections of $\{\Omega_l\}$ and $\{u_l,v_l\}$.

For $u,v\in C^{\infty}(U^{\delta})$,  we choose a sequence of
relatively compact open sets $V_l\uparrow U^{\delta}$ and a
sequence of functions $\{u_l,v_l\}\subset
C^{\infty}_0(U^{\delta})$ such that $u=u_l$ and $v=v_l$  on $V_l$.
By (\ref{e01}) and the left strong local property  of
${\mathcal{E}}^{c,\delta}$ and $\hat{\mathcal{E}}^{c,\delta}$, we
can well define the linear functional $L^{\delta}(u,v)$ by
$<L^{\delta}(u,v),f>=\lim_{l\rightarrow\infty}<L^{\delta}(u_l,v_l),f>$
for $f\in C^{\infty}_0(U^{\delta})$. The definition of
$L^{\delta}(u,v)$ is independent of the selections of $\{V_l\}$
and $\{u_l,v_l\}$.

\begin{thm}\label{t28} Let $\Phi\in C^{\infty}(\mathbb{R}^m)$.

\noindent (i) For $u_1, \dots, u_m,v,w\in C^{\infty}(U)$,
$$d\mu^c_{<\Phi(u_1,\dots,u_m),v>}=\sum_{i=1}^m\Phi_{x_i}(u_1,\dots,u_m)d\mu^c_{<u_i,v>}.
$$

\noindent (ii) For $u_1, \dots, u_m,v,w\in C^{\infty}(U^{\delta})$
and $f\in C_0^{\infty}(U^{\delta})$,
$$<L^{\delta}(\Phi(u_1,\dots,u_m),vw),f>=\sum_{i=1}^m<L^{\delta}(u_i,v),\Phi_{x_i}(u_1,\dots,u_m)wf>.
$$
\end{thm}

\begin{proof} Since the constant function
belongs to $C^{\infty}(U)$, to prove the theorem, we may assume
without loss of generality that $\Phi\in C^{\infty}(\mathbb{R}^m)$
with $\Phi(0)=0$ and $u_1,\dots,u_m,v,w,f\in
C_{0}^{\infty}(U^{\delta})$.  To simplify notation, we denote
$u=(u_1,\dots,u_m)$. Let ${\cal C}$ be the family of all $\Phi$
satisfying (i) and (ii). By Theorem \ref{thm9}, we know that if
$\Psi,\Gamma\in{\cal C}$, then $\Psi\Gamma\in {\cal C}$. Since
${\cal C}$ contains the coordinate functions, it contain all
polynomials vanishing at the origin.

Let $V$ be a finite cube containing the range of $u$.  Then, there
exists a sequence $\{\Phi^{(k)}\}$ of polynomials vanishing at the
origin such that $\Phi^{(k)}$ and all of its partial derivatives
converge uniformly to $\Phi$ and its corresponding partial
derivatives on $V$ (cf. \cite[II \S4]{CH53}).

(i) Let $g\in C_{0}^{\infty}(U^{\delta})$. We choose a $\phi\in
C^{\infty}_{0}(U)$ satisfying $\phi=1$ on
$F^{\delta}_{u_1}\cup\cdots\cup F^{\delta}_{u_m}\cup
F^{\delta}_v\cup F^{\delta}_{g}$ (see (\ref{eqa}) for the definition of $F^{\delta}_{\cdot}$). Then, we obtain by (\ref{znew}),
(\ref{ee1}), the assumption that $C^{\infty}_{0}(U)\subset D(A)\cap D(\hat A)$ or Assumption \ref{as1}, Taylor's theorem and Lemma \ref{l1}(ii), the
finiteness of $J$ on $({\rm supp}[\phi]\times {\rm
supp}[\phi])\cap\{|x-y|>\delta\}$, and the dominated convergence
theorem that
\begin{eqnarray*}
& &\int_{U^{\delta}}gd\mu^c_{<\Phi(u),v>}\\
&=&{\mathcal{E}}^{c,\delta}(\Phi(u),vg)+{\mathcal{E}}^{c,\delta}(v,\Phi(u)g)-{\mathcal{E}}^{c,\delta}(\Phi(u)v,g)\\
&=&{\cal E}(\Phi(u),vg)+{\cal E}(v,\Phi(u)g)-{\cal E}(\Phi(u)v,g)-{\cal E}(\phi,\Phi(u)vg)\\
&-&\int_{U\times U\backslash d}2\left(\Phi(u)(y)-\Phi(u)(x)-\sum_{i=1}^{n}(y_{i}-x_{i})\frac{\partial \Phi(u)}{\partial y_{i}}(y)I_{\{|x-y|\le\delta\}}(x,y)\right)(vg)(y)\phi(x)J(dxdy)\\
&-&\int_{U\times U\backslash d}2\left(v(y)-v(x)-\sum_{i=1}^{n}(y_{i}-x_{i})\frac{\partial v}{\partial y_{i}}(y)I_{\{|x-y|\le\delta\}}(x,y)\right)(\Phi(u)g)(y)\phi(x)J(dxdy)\\
&+&\int_{U\times U\backslash d}2\left((\Phi(u)v)(y)-(\Phi(u)v)(x)-\sum_{i=1}^{n}(y_{i}-x_{i})\frac{\partial (\Phi(u)v)}{\partial y_{i}}(y)I_{\{|x-y|\le\delta\}}(x,y)\right)\\
& &\ \ \ \ \cdot g(y)\phi(x)J(dxdy)\\
&=&\lim_{k\rightarrow\infty}\left\{{\cal E}(\Phi^{(k)}(u),vg)+{\cal E}(v,\Phi^{(k)}(u)g)-{\cal E}(\Phi^{(k)}(u)v,g)-{\cal E}(\phi,\Phi^{(k)}(u)vg)\right.\\
&-&\int_{U\times U\backslash d}2\left(\Phi^{(k)}(u)(y)-\Phi^{(k)}(u)(x)-\sum_{i=1}^{n}(y_{i}-x_{i})\frac{\partial \Phi^{(k)}(u)}{\partial y_{i}}(y)I_{\{|x-y|\le\delta\}}(x,y)\right)\\
& &\ \ \ \ \cdot (vg)(y)\phi(x)J(dxdy)\\
&-&\int_{U\times U\backslash d}2\left(v(y)-v(x)-\sum_{i=1}^{n}(y_{i}-x_{i})\frac{\partial v}{\partial y_{i}}(y)I_{\{|x-y|\le\delta\}}(x,y)\right)(\Phi^{(k)}(u)g)(y)\phi(x)J(dxdy)\\
&+&\left.\int_{U\times U\backslash d}2\left((\Phi^{(k)}(u)v)(y)-(\Phi^{(k)}(u)v)(x)\right.\right.\\
& &\ \ \ \ \ \ \ \ \ \ \ \ \ -\sum_{i=1}^{n}(y_{i}-x_{i})\frac{\partial (\Phi^{(k)}(u)v)}{\partial y_{i}}(y)I_{\{|x-y|\le\delta\}}(x,y))g(y)\phi(x)J(dxdy)\}\\
&=&\lim_{k\rightarrow\infty}\{{\mathcal{E}}^{c,\delta}(\Phi^{(k)}(u),vg)+{\mathcal{E}}^{c,\delta}(v,\Phi^{(k)}(u)g)-{\mathcal{E}}^{c,\delta}(\Phi^{(k)}(u)v,g)\}\\
&=&\lim_{k\rightarrow\infty}\int_{U^{\delta}}gd\mu^c_{<\Phi^{(k)}(u),v>}\\
&=&\lim_{k\rightarrow\infty}\sum_{i=1}^m\int_{U^{\delta}}g\Phi^{(k)}_{x_i}(u)d\mu^c_{<u_i,v>}\\
&=&\sum_{i=1}^m\int_{U^{\delta}}g\Phi_{x_i}(u)d\mu^c_{<u_i,v>}.
\end{eqnarray*}

(ii) We choose a $\phi\in C^{\infty}_{0}(U)$ satisfying $\phi=1$
on $F^{\delta}_{u_1}\cup\cdots\cup F^{\delta}_{u_m}\cup
F^{\delta}_v\cup F^{\delta}_{f}$. By (\ref{znew}), (\ref{z1}),
(\ref{e01}), the assumption that $C^{\infty}_{0}(U)\subset
D(A)\cap D(\hat A)$ or Assumption \ref{as1}, Taylor's theorem and
Lemma \ref{l1}(ii), the finiteness of $J$ on $({\rm
supp}[\phi]\times {\rm supp}[\phi])\cap\{|x-y|>\delta\}$, and the
dominated convergence theorem, we get
\begin{eqnarray*}
& &<L^{\delta}(\Phi(u),v),f>\\
&=&\frac{1}{2}({\mathcal{E}}^{c,\delta}(\Phi(u),vf)-\hat{\mathcal{E}}^{c,\delta}(\Phi(u),vf))\\
&=&\frac{1}{2}\left[\mathcal{E}(\Phi(u),vf)-{\cal E}(\phi,\Phi(u)vf)-\hat\mathcal{E}(\Phi(u),vf)+\hat{\cal E}(\phi,\Phi(u)vf)\right]\\
&- &\int_{U\times U\backslash d}\left(\Phi(u)(y)-\Phi(u)(x)-\sum_{i=1}^{n}(y_{i}-x_{i})\frac{\partial \Phi(u)}{\partial y_{i}}(y)I_{\{|x-y|\le\delta\}}(x,y)\right)(vf)(y)\phi(x)J(dxdy)\\
&+ &\int_{U\times U\backslash d}\left(\Phi(u)(y)-\Phi(u)(x)-\sum_{i=1}^{n}(y_{i}-x_{i})\frac{\partial \Phi(u)}{\partial y_{i}}(y)I_{\{|x-y|\le\delta\}}(x,y)\right)(vf)(y)\phi(x)\hat J(dxdy)\\
&=&\lim_{k\rightarrow\infty}\left\{\frac{1}{2}\left[\mathcal{E}(\Phi^{(k)}(u),vf)-{\cal
E}(\phi,\Phi^{(k)}(u)vf)
-\hat\mathcal{E}(\Phi^{(k)}(u),vf)+\hat{\cal E}(\phi,\Phi^{(k)}(u)vf)\right]\right.\\
&- &\int_{U\times U\backslash d}\left(\Phi^{(k)}(u)(y)-\Phi^{(k)}(u)(x)\right.\\
& &\ \ \ \ \ \ \ \ \ \ \ \ \ \left.-\sum_{i=1}^{n}(y_{i}-x_{i})\frac{\partial \Phi^{(k)}(u)}{\partial y_{i}}(y)I_{\{|x-y|\le\delta\}}(x,y))(vf)(y)\phi(x)J(dxdy)\right.\\
&+ &\int_{U\times U\backslash d}\left(\Phi^{(k)}(u)(y)-\Phi^{(k)}(u)(x)\right.\\
& &\ \ \ \ \ \ \ \ \ \ \ \ \ \left.-\sum_{i=1}^{n}(y_{i}-x_{i})\frac{\partial \Phi^{(k)}(u)}{\partial y_{i}}(y)I_{\{|x-y|\le\delta\}}(x,y))(vf)(y)\phi(x)\hat J(dxdy)\right\}\\
&=&\lim_{k\rightarrow\infty}\frac{1}{2}({\mathcal{E}}^{c,\delta}(\Phi^{(k)}(u),vf)-\hat{\mathcal{E}}^{c,\delta}(\Phi^{(k)}(u),vf))\\
&=&\lim_{k\rightarrow\infty}<L^{\delta}(\Phi^{(k)}(u),v),f>\\
&=&\lim_{k\rightarrow\infty}\sum_{i=1}^m<L^{\delta}(u_i,v),\Phi_{x_i}^{(k)}(u)f>\\
&=&\lim_{k\rightarrow\infty}\sum_{i=1}^m\frac{1}{2}({\mathcal{E}}^{c,\delta}(u_i,v\Phi_{x_i}^{(k)}(u)f)-\hat{\mathcal{E}}^{c,\delta}(u_i,v\Phi_{x_i}^{(k)}(u)f))\\
&=&\lim_{k\rightarrow\infty}\sum_{i=1}^m\left\{\frac{1}{2}[\mathcal{E}(u_i,v\Phi_{x_i}^{(k)}(u)f)-{\cal
E}(\phi,u_iv\Phi_{x_i}^{(k)}(u)f)-\hat\mathcal{E}(u_i,v\Phi_{x_i}^{(k)}(u)f)+\hat{\cal E}(\phi,u_iv\Phi_{x_i}^{(k)}(u)f)]\right.\\
&-&\int_{U\times U\backslash d}\left(u_i(y)-u_i(x)\right.\\
& &\ \ \ \ \ \ \ \ \ \ \ \ \ \left.-\sum_{j=1}^{n}(y_{j}-x_{j})\frac{\partial u_i}{\partial y_{j}}(y)I_{\{|x-y|\le\delta\}}(x,y))(v\Phi_{x_i}^{(k)}(u)f)(y)\phi(x)J(dxdy)\right.\\
&+ &\int_{U\times U\backslash d}\left(u_i(y)-u_i(x)\right.\\
& &\ \ \ \ \ \ \ \ \ \ \ \ \ \left.-\sum_{j=1}^{n}(y_{j}-x_{j})\frac{\partial u_i}{\partial y_{j}}(y)I_{\{|x-y|\le\delta\}}(x,y))(v\Phi_{x_i}^{(k)}(u)f)(y)\phi(x)\hat J(dxdy)\right\}\\
&=&\sum_{i=1}^m\left\{\frac{1}{2}[\mathcal{E}(u_i,v\Phi_{x_i}(u)f)-{\cal
E}(\phi,u_iv\Phi_{x_i}(u)f)-\hat\mathcal{E}(u_i,v\Phi_{x_i}(u)f)+\hat{\cal E}(\phi,u_iv\Phi_{x_i}(u)f)]\right.\\
&- &\int_{U\times U\backslash d}\left(u_i(y)-u_i(x)\right.\\
& &\ \ \ \ \ \ \ \ \ \ \ \ \ \left.-\sum_{j=1}^{n}(y_{j}-x_{j})\frac{\partial u_i}{\partial y_{j}}(y)I_{\{|x-y|\le\delta\}}(x,y))(v\Phi_{x_i}(u)f)(y)\phi(x)J(dxdy)\right.\\
&+ &\int_{U\times U\backslash d}\left(u_i(y)-u_i(x)\right.\\
& &\ \ \ \ \ \ \ \ \ \ \ \ \ \left.-\sum_{j=1}^{n}(y_{j}-x_{j})\frac{\partial u_i}{\partial y_{j}}(y)I_{\{|x-y|\le\delta\}}(x,y))(v\Phi_{x_i}(u)f)(y)\phi(x)\hat J(dxdy)\right\}\\
&=&\sum_{i=1}^m\frac{1}{2}({\mathcal{E}}^{c,\delta}(u_i,v\Phi_{x_i}(u)f)-\hat{\mathcal{E}}^{c,\delta}(u_i,v\Phi_{x_i}(u)f))\\
&=&\sum_{i=1}^m<L^{\delta}(u_i,v),\Phi_{x_i}(u)f>.
\end{eqnarray*}
Therefore, the proof is complete by noting Theorem \ref{thm9}(ii).
\end{proof}

\subsection{Proofs of Theorems \ref{new90} and \ref{lll}}
\begin{proof}[\textbf{Proofs of Theorems \ref{new90} and \ref{lll}}] We first characterize the
first two terms of (\ref{bb1}). Suppose that $u,v\in
C^{\infty}_{0}(U^{\delta})$ and $\chi\in
C^{\infty}_{0}(U^{\delta})$ satisfying $\chi=1$  on a
neighbourhood of  ${\rm supp}[u]\cup{\rm supp}[v]$. Denote by
$x_i$, $1\le i\le n$, the coordinate functions of $\mathbb{R}^n$.
For $1\le i,j\le n$, we define $\nu_{ij}:=\mu^c_{<x_i,x_j>}$,
which is a Radon measure on $U$. Then, by Theorem \ref{t28}(i), we
get
\begin{equation}\label{xc1}
\int_{U}\chi d\mu_{<u,v>}^c=\sum_{i,j=1}^n\int_{U}\frac{\partial
u}{\partial x_i}\frac{\partial v}{\partial x_j}\nu_{ij}(dx).
\end{equation}

For $1\le i\le n$, we define the linear functional $\Psi^{\delta}_i$
on $C^{\infty}_{0}(U^{\delta})$ by
\begin{equation}\label{xc3}
<\Psi^{\delta}_i, f>=<L^{\delta}(x_i,1),f>,\ \ f\in
C^{\infty}_{0}(U^{\delta}).
\end{equation}
Then, by Theorem \ref{t28}(ii) and (\ref{xc3}), we get
\begin{equation}\label{xc2}
<L^{\delta}(u,v),\chi>=\sum_{i=1}^n\left<\Psi^{\delta}_i,
\frac{\partial u}{\partial x_i}v\right>.
\end{equation}

We now show that each $\Psi^{\delta}_i$ is a generalized function on
$U^{\delta}$. Let $O$ be an arbitrary relatively compact open set
of $U^{\delta}$. Suppose that $\{f_n\}$ is a sequence of functions
in $C^{\infty}_{0}(O)$ such that $f_n$ and all of its partial
derivatives converge uniformly to some $f\in C^{\infty}_{0}(O)$
and its corresponding partial derivatives as $n\rightarrow\infty$.
We fix a $\xi_i\in C^{\infty}_0(U^{\delta})$ satisfying
$\xi_i=x_i$ on $O$ and choose a $\psi\in C^{\infty}_{0}(U)$
satisfying $\psi=1$ on $F^{\delta}_{\xi_i}\cup \{x\in U:\inf_{y\in
O}|x-y|\le\delta\}$. For $g\in C^{\infty}_{0}(O)$, by
(\ref{znew}), (\ref{z1}), (\ref{e01}) and (\ref{xc3}), we get
\begin{eqnarray}\label{FG1}
& &<\Psi^{\delta}_i, g>\nonumber\\
&=&<L^{\delta}(x_i,1),g>\nonumber\\
&=&\frac{1}{2}({\mathcal{E}}^{c,\delta}(\xi_i,g)-\hat{\mathcal{E}}^{c,\delta}(\xi_i,g))\nonumber\\
&=&\frac{1}{2}(\mathcal{E}(\xi_i,g)-{\cal E}(\psi,\xi_i g)-\hat\mathcal{E}(\xi_i,g)+\hat {\cal E}(\psi,\xi_i g))\nonumber\\
&&-\int_{U\times U\backslash d}\left(\xi_i(y)-\xi_i(x)-(y_i-x_i)I_{\{|x-y|\le\delta\}}(x,y)\right)g(y)\psi(x) J(dxdy)\nonumber\\
&&+\int_{U\times U\backslash
d}\left(\xi_i(y)-\xi_i(x)-(y_i-x_i)I_{\{|x-y|\le\delta\}}(x,y)\right)g(y)\psi(x)\hat
J(dxdy).\ \ \ \ \ \ \ \
\end{eqnarray}
Then, we obtain by (\ref{FG1}), the assumption that
$C^{\infty}_{0}(U)\subset D(A)\cap D(\hat A)$ or Assumption
\ref{as1}, Taylor's theorem and Lemma \ref{l1}(ii), the finiteness
of $J$ on $({\rm supp}[\psi]\times {\rm
supp}[\psi])\cap\{|x-y|>\delta\}$, and the dominated convergence
theorem that $<\Psi^{\delta}_i,
f>=\lim_{n\rightarrow\infty}<\Psi^{\delta}_i, f_n>$. Therefore, the
proof of Theorem \ref{lll} is complete by (\ref{E102}),
(\ref{bb1}), (\ref{xc1}) and (\ref{xc2}).

To complete the proof of Theorem \ref{new90}, we need only show
that there exist signed Radon measures
$\{\nu^{\delta}_i\}_{i=1}^n$ on $U^{\delta}$ such that for each
$1\le i\le n$,
$$
<\Psi^{\delta}_i, g>=\int_{U^{\delta}}g(x)\nu^{\delta}_i(dx),\ \
\forall g\in C^{\infty}_{0}(U^{\delta}).
$$
In fact, let $O$ be an arbitrary relatively compact open set of $U^{\delta}$. Then, by (\ref{FG1}), the assumption that $C^{\infty}_{0}(U)\subset D(A)\cap D(\hat A)$ and Lemma \ref{l1} (ii) and (iii), one finds that there exists a unique signed Radon measure $\nu_i^O$ on $O$ such that
$$
<\Psi^{\delta}_i, g>=\int_{O}g(x)\nu^O_i(dx),\ \ \forall g\in
C^{\infty}_{0}(O).
$$
Therefore, we can well define $\nu^{\delta}_i=\nu_i^O$ for each
$O$. The proof is complete.
\end{proof}

From now on till the end of this section, we suppose that $({\cal E}, D({\cal E}))$ is a semi-Dirichlet form on $L^2(U;m)$ satisfying $C^{\infty}_0(U)\subset D({\cal E})$.

\begin{rem}\label{rem2} Assumption \ref{as1} is implied by the following
assumption.
\begin{assum}\label{as2}
There exist a sequence of  Dirichlet forms $({\cal Q}^l, D({\cal
Q}^l))$ on $L^2(\Omega_l;m)$ and a sequence of positive constants $C_l$
such that $C^{\infty}_0(\Omega_l)\subset D({\cal Q}_l)$ and
$$
{\cal E}_1(g,g)\le C_l{\cal Q}^l_1(g,g), \ \ \forall g\in
C^{\infty}_0(\Omega_l).
$$
\end{assum}

In fact, suppose Assumption \ref{as2} holds and $O$ is a
relatively compact open set of $U^{\delta}$. Then, there exist an
open set $O_0$ of $U^{\delta}$ satisfying $\overline{O}\subset
O_0$  and a regular symmetric Dirichlet form $({\cal Q}, D({\cal
Q}))$ on $L^2(O_0;m)$ such that $C^{\infty}_0(O_0)\subset D({\cal
Q})$ and
\begin{equation}\label{GH1}
{\cal E}_1(g,g)\le C{\cal Q}_1(g,g), \ \ \forall g\in
C^{\infty}_0(O_0),
\end{equation}
for some positive constant $C$. We consider the classical
Beurling-Deny formula for $({\cal Q}, D({\cal Q}))$ (cf.
\cite[Theorem 3.2.3]{Fu94}):
\begin{eqnarray}\label{new2}
& &{\cal Q}(u,v)=\frac{1}{2}\sum_{i,j=1}^n\int_{O_0}\frac{\partial u}{\partial x_i}\frac{\partial v}{\partial x_j}\nu^{\cal Q}_{ij}(dx)\nonumber\\
& &\ \ \ \ +\int_{O_0\times O_0\backslash
d}(u(x)-u(y))(v(x)-v(y))J^{\cal Q}(dxdy)+\int_{O_0}u(x)v(x)K^{\cal
Q}(dx),\ \ \ \ \ \ \ \ \ \ \
\end{eqnarray}
where $u,v\in C^{\infty}_0(O_0)$ and we use the superscript
``${\cal Q}$" to emphasize that the corresponding Radon measures
are for $({\cal Q}, D({\cal Q}))$. Note that for any compact set
$K$ and open set $O_1$ with $K\subset O_1\subset O_0$ (cf.
\cite[(1.2.4)]{Fu94}),
\begin{equation}\label{nb11}
\int_{K\times K\backslash d}|x-y|^2J^{\cal Q}(dxdy)<\infty,\ \
J^{\cal Q}(K,O_0-O_1)<\infty.
\end{equation}

Suppose $\{f_n\}_{n=1}^{\infty}\subset C^{\infty}_{0}(O)$ and
$f\in C^{\infty}_{0}(O)$ satisfying $f_n$ and all of its partial
derivatives converge uniformly to $f$ and its corresponding
partial derivatives as $n\rightarrow\infty$. By (\ref{new2}),
(\ref{nb11}) and the dominated convergence theorem,  we find that
$f_n$ converges to $f$ w.r.t. the $\tilde{\cal Q}_{1}^{1/2}$-norm
as $n\rightarrow\infty$. Therefore, we obtain by (\ref{GH1}) that
$\lim_{n\rightarrow\infty}{\cal E}_1(f_n-f,f_n-f)=0$.

\end{rem}

\begin{cor}\label{dfg}
Assume the setting of Theorem \ref{lll} but with Assumption \ref{as1} replaced by Assumption \ref{as2}.
Then, we have the decomposition given in Theorem \ref{lll}. Moreover, for any relatively compact open set $O$ of $U^{\delta}$, there exist signed Radon measures $\{\mu^O_i\}_{i=1}^n$ and $\{\mu^O_{ij}\}_{i,j=1}^n$ on $O$ such that for $1\le i\le n$,
$$
<\Psi^{\delta}_i,
g>=\int_{O}g(x)\mu^O_i(dx)+\sum_{j=1}^n\int_{O}\frac{\partial
g}{\partial x_j}(x)\mu^O_{ij}(dx),\ \ \forall g\in
C^{\infty}_{0}(O).
$$
\end{cor}
\begin{proof}
Let $O$ be a relatively compact open set of $U^{\delta}$. By Assumption \ref{as2}, there exist an
open set $O_0$ of $U^{\delta}$ satisfying $\overline{O}\subset
O_0$  and a regular symmetric Dirichlet form $({\cal Q}, D({\cal
Q}))$ on $L^2(O_0;m)$ such that $C^{\infty}_0(O_0)\subset D({\cal
Q})$ and (\ref{GH1}) holds.

By (\ref{FG1}), (\ref{GH1}), the sector condition and Lemma \ref{l1} (ii) and (iii), to prove the corollary, we need only show that for any $u\in D({\cal
Q})$ there exist signed Radon measures $\mu^u$ and $\{\mu^u_{j}\}_{j=1}^n$ on $O$ such that
$$
{\cal Q}(u,v)=\int_{O}v(x)\mu^u(dx)+\sum_{j=1}^n\int_{O}\frac{\partial v}{\partial x_j}(x)\mu^u_{j}(dx),\ \ \forall v\in C^{\infty}_{0}(O).
$$
By \cite[Theorems 3.2.2 and 5.3.1]{Fu94}, we get
\begin{eqnarray*}
& &{\cal Q}(u,v)=\frac{1}{2}\sum_{j=1}^n\int_{O_0}\frac{\partial v}{\partial x_j}\mu^c_{<u,\xi_j>}(dx)\nonumber\\
& &\ \ \ \ +\int_{O_0\times O_0\backslash
d}(\tilde u(x)-\tilde u(y))(v(x)-v(y))J^{\cal Q}(dxdy)+\int_{O_0}u(x)v(x)K^{\cal
Q}(dx),\ \ \ \ \ \ \ \ \ \ \
\end{eqnarray*}
where $\xi_j\in C^{\infty}_0(U^{\delta})$ satisfying $\xi_j=x_j$ on $O$ for $1\le j\le n$ as in (\ref{FG1}), $\mu^c$ denotes the local part of the energy measure of $({\cal Q}, D({\cal
Q}))$, $\tilde u$ denotes a quasi-continuous version of $u$. Therefore, the proof is complete by the mean value theorem, (\ref{nb11}) and the Riesz representation theorem.
\end{proof}

\section [short title]{LeJan type transformation rule for the diffusion part of
regular semi-Dirichlet forms}
\setcounter{equation}{0}

In this section, we will apply some ideas of  Section 2 to investigate the structure of general regular semi-Dirichlet forms. Throughout this section, we let $E$ be a locally compact separable metric space, $m$ a positive Radon measure on $E$ with ${\rm supp}[m]=E$, and $({\cal E}, D({\cal E}))$ a regular semi-Dirichlet form on $L^2(E;m)$.

Following Section 1, we use $J$ and $K$ to denote respectively the jumping and killing measures of $({\cal E}, D({\cal E}))$. By \cite[Corollary 2.2]{HC06}, there exists a unique positive Radon measure $\sigma_{\beta}$ on $E\times E$ satisfying
$$
(\beta G_{\beta}u,v)=\int_{E\times E}u(x)v(y)\sigma_{\beta}(dxdy)\ \ {\rm for}\ u,v\in L^2(E;m).
$$
Hereafter $(\cdot,\cdot)$ denotes the inner product of $L^2(E;m)$ and $(G_{\beta})_{\beta>0}$ denotes the resolvent of $({\cal E}, D({\cal E}))$. We have $(\beta/2)\sigma_{\beta}\rightarrow J$ vaguely on $E\times E\backslash d$ as $\beta\rightarrow\infty$ (cf. the proof of \cite[Theorem 2.6]{HC06}). Define $\hat J(dxdy):=J(dydx)$, $\hat\sigma_{\beta}(dxdy):=\sigma_{\beta}(dydx)$, and denote by $(\hat G_{\beta})_{\beta>0}$ the co-resolvent of $({\cal E}, D({\cal E}))$. Then, we have
$$
(\beta \hat G_{\beta}u,v)=\int_{E\times E}u(x)v(y)\hat\sigma_{\beta}(dxdy)\ \ {\rm for}\ u,v\in L^2(E;m)
$$
and $(\beta/2)\hat\sigma_{\beta}\rightarrow \hat J$ vaguely on $E\times E\backslash d$ as $\beta\rightarrow\infty$.

Let $\rho$ be the metric on $E$. We choose a sequence of
relatively compact open sets $\Omega_l\uparrow E$ and a  sequence
of numbers $\varsigma_l\downarrow 0$ such that the set
$\Gamma_l=\{(x,y)\in \Omega_l\times \Omega_l:\rho(x,y)\ge
\varsigma_l\}$ is a continuous set w.r.t. $J$ for every $l\in
\mathbb{N}$. Denote $\Lambda_l=\{(x,y)\in \Omega_l\times
\Omega_l:\rho(x,y)<\varsigma_l\}$.

We make the following assumption.
\begin{assum}\label{as33}
For $f,g\in C_{0}(E)\cap D(\mathcal{E})$, we have $fg\in
C_{0}(E)\cap D(\mathcal{E})$ and $(f(y)-f(x))g(y)$ is integrable
w.r.t. $J$.
\end{assum}

Suppose $u,v\in C_{0}(E)\cap D(\mathcal{E})$. Let $\chi\in C_{0}(E)\cap D(\mathcal{E})$ satisfying $\chi=1$ on a neighbourhood of
${\rm supp}[u]\cup{\rm supp}[v]$. Then, by Assumption \ref{as33}, we get
\begin{eqnarray}\label{v1}
& &\hat{\mathcal{E}}(u,v)=\lim_{\beta\rightarrow\infty}\beta(u-\beta\hat {G}_{\beta}u,v)\nonumber\\
&=&\lim_{\beta\rightarrow\infty}\beta\left\{\int_{E\times E}(u(y)-u(x))v(y)\chi(x)\hat{\sigma}_{\beta}(dxdy)\right.\nonumber\\
& &\ \ \ \ +\left.\int_{E}\chi(x)u(x)v(x)m(dx)-\int_{E\times E}\chi(x)u(y)v(y)\hat{\sigma}_{\beta}(dxdy)\right\}\nonumber\\
&=&\lim_{\beta\rightarrow\infty}\beta\int_{E\times E}(u(y)-u(x))v(y)\chi(x)\hat{\sigma}_{\beta}(dxdy)+\hat{\cal E}(\chi,uv)\nonumber\\
&=&\lim_{l\rightarrow\infty}\lim_{\beta\rightarrow\infty}\beta\int_{\Lambda_l}(u(y)-u(x))v(y)\hat\sigma_{\beta}(dxdy)\nonumber\\
& &+\int_{E\times E\backslash d}2(u(y)-u(x))v(y)\chi(x)\hat J(dxdy)+\hat{\cal E}(\chi,uv).
\end{eqnarray}
Hence we can well define
\begin{eqnarray}\label{v2}
\hat{\mathcal{E}}^{c}(u,v):=\lim_{l\rightarrow\infty}\lim_{\beta\rightarrow\infty}\beta\int_{\Lambda_l}
(u(y)-u(x))v(y)\hat\sigma_{\beta}(dxdy).
\end{eqnarray}
$\hat{\mathcal{E}^{c}}(u,v)$ satisfies the left strong local property in the sense that $\hat{\mathcal{E}^{c}}(u,v)=0$ whenever $u$ is constant on a neighbourhood of ${\rm supp}[v]$. By (\ref{v1}) and (\ref{v2}), we obtain the decomposition
\begin{eqnarray}\label{BD100}
\hat\mathcal{E}(u,v)=\hat\mathcal{E}^{c}(u,v)+\int_{E\times E\backslash d}2(u(y)-u(x))v(y)\chi(x)\hat J(dxdy)+\hat {\cal E}(\chi,uv).
\end{eqnarray}

Similar to (\ref{v1}), we can show that
\begin{eqnarray}\label{vv1}
{\mathcal{E}}(u,v)&=&\lim_{l\rightarrow\infty}\lim_{\beta\rightarrow\infty}\beta\int_{\Lambda_l}(u(y)-u(x))v(y)\sigma_{\beta}(dxdy)\nonumber\\
& &+\int_{E\times E\backslash d}2(u(y)-u(x))v(y)\chi(x)J(dxdy)+{\cal E}(\chi,uv).
\end{eqnarray}
By (\ref{BD2}) and (\ref{vv1}), we get
\begin{eqnarray}\label{vv11}
{\mathcal{E}}^c(u,v)&=&\lim_{l\rightarrow\infty}\lim_{\beta\rightarrow\infty}\beta\int_{\Lambda_l}(u(y)-u(x))v(y)\sigma_{\beta}(dxdy)\nonumber\\
& &+\int_{E\times E\backslash d}2(u(y)-u(x))v(y)\chi(x)J(dxdy)\nonumber\\
& &+\int_{E\times E\backslash d}2(\chi(y)-\chi(x))(uv)(y)J(dxdy)+\int_Eu(x)v(x)K(dx)\nonumber\\
& &-\int_{E\times E\backslash d}2(u(y)-u(x))v(y)J(dxdy)-\int_Eu(x)v(x)K(dx)\nonumber\\
&=&\lim_{l\rightarrow\infty}\lim_{\beta\rightarrow\infty}\beta\int_{\Lambda_l}(u(y)-u(x))v(y)\sigma_{\beta}(dxdy).
\end{eqnarray}

For $r\in\mathbb{N}$, we choose a $w\in C_{0}(E)\cap
D(\mathcal{E})$ satisfying $w\ge 0$ and $w|_{\Omega_r}\equiv 1$.
For $f\in C_{0}(\Omega_r)\cap D(\mathcal{E})$, we obtain by
(\ref{vv11}) and the sub-Markovian property of
$(G_{\beta})_{\beta>0}$ that
\begin{eqnarray*}
& &|2\mathcal{E}^{c}(u,uf)-\mathcal{E}^{c}(u^{2},f)|\\
&=&\left|\lim_{l\rightarrow\infty}\lim_{\beta\rightarrow\infty}\beta\int_{\Lambda_l}(u(y)-u(x))^2f(y)\sigma_{\beta}(dxdy)\right|\\
&\le&\|f\|_{\infty}\lim_{\beta\rightarrow\infty}\beta\int_{E\times E}(u(y)-u(x))^2w(y)\sigma_{\beta}(dxdy)\\
&\le&\|f\|_{\infty}\lim_{\beta\rightarrow\infty}\{2\beta(u-\beta G_{\beta}u,uw)-\beta(u^2-\beta G_{\beta}u^2,w)\}\\
&=&(2{\cal E}(u,uw)-{\cal E}(u^2,w))\|f\|_{\infty}.
\end{eqnarray*}
Then, there exists a unique Radon measure $\mu_{<u>}^{r,c}$ on
$\Omega_r$ such that
$$
\int_{\Omega_r}fd\mu_{<u>}^{r,c}=2\mathcal{E}^{c}(u,uf)-\mathcal{E}^{c}(u^{2},f),\
\ \forall f\in C_{0}(\Omega_r)\cap D(\mathcal{E}).
$$
It is easy to see that $\{\mu_{<u>}^{r,c}\}$ is a consistent
sequence of Radon measures. Therefore, we can well define the
measure $\mu_{<u>}^{c}$ by $\mu_{<u>}^{c}=\mu_{<u>}^{r,c}$ on
$\Omega_r$, which satisfies
$$
\int_{E}fd\mu_{<u>}^{c}=2\mathcal{E}^{c}(u,uf)-\mathcal{E}^{c}(u^{2},f),\
\ \forall f\in C_{0}(E)\cap D(\mathcal{E}).
$$

We define
$$
\mu_{<u,v>}^c:=\frac{1}{2}(\mu_{<u+v>}^c-\mu_{<u>}^c-\mu_{<v>}^c).
$$
Then
$$
\int_{E}fd\mu_{<u,v>}^c=\mathcal{E}^{c}(u,vf)+\mathcal{E}^{c}(v,uf)-\mathcal{E}^{c}(uv,f),\ \ f\in C_{0}(E)\cap D(\mathcal{E}).
$$
Hence, for any $h\in C_{0}(E)\cap D(\mathcal{E})$ satisfying $h|_{{\rm supp}[u]\cup{\rm supp}[v]}\equiv1$, we have
\begin{equation}\label{bv1v}
\mathcal{E}^{c}(u,v)+\mathcal{E}^{c}(v,u)=\int_{E}hd\mu_{<u,v>}^c+\mathcal{E}^{c}(uv,h).
\end{equation}

We define a linear functional $L(u,v)$ on $C_{0}(E)\cap D(\mathcal{E})$ by
$$<L(u,v),f>:=\frac{1}{2}(\mathcal{E}^{c}(u,vf)-\hat{\mathcal{E}}^{c}(u,vf)),\ \ f\in C_{0}(E)\cap D(\mathcal{E}).
$$
Then, for any $h\in C_{0}(E)\cap D(\mathcal{E})$ satisfying $h|_{{\rm supp}[u]}\equiv1$, we have
\begin{equation}\label{bv1}
\mathcal{E}^{c}(u,v)-\hat\mathcal{E}^{c}(u,v)=2<L(u,v),h>.
\end{equation}
By (\ref{BD2}), (\ref{BD100}) and (\ref{bv1}), we get
\begin{eqnarray}\label{eee4}
& &\mathcal{E}^{c}(u,v)-\mathcal{E}^{c}(v,u)\nonumber\\
&=&\mathcal{E}^{c}(u,v)-\mathcal{E}(v,u)+\int_Eu(x)v(x)K(dx)+\int_{E\times E\backslash d}2(v(y)-v(x))u(y)J(dxdy)\nonumber\\
&=&\mathcal{E}^{c}(u,v)-\hat\mathcal{E}^{c}(u,v)-\hat {\cal E}(\chi,uv)+\int_Eu(x)v(x)K(dx)\nonumber\\
&&-\int_{E\times E\backslash d}2(u(y)-u(x))v(y)\chi(x)\hat J(dxdy)+\int_{E\times E\backslash d}2(v(y)-v(x))u(y)J(dxdy)\nonumber\\
&=&2<L(u,v),\chi>-{\cal E}(uv,\chi)+\int_Uu(x)v(x)K(dx)\nonumber\\
&&-\int_{E\times E\backslash d}2(u(y)-u(x))v(y)\chi(x)\hat J(dxdy)+\int_{E\times E\backslash d}2(v(y)-v(x))u(y)J(dxdy)\nonumber\\
&=&2<L(u,v),\chi>-{\cal E}^{c}(uv,\chi)\nonumber\\
&&-\int_{E\times E\backslash d}2((uv)(y)-(uv)(x))\chi(y)J(dxdy)\nonumber\\
&&-\int_{E\times E\backslash d}2(u(x)-u(y))v(x)\chi(y)J(dxdy)+\int_{E\times E\backslash d}2(v(y)-v(x))u(y)J(dxdy)\nonumber\\
&=&2<L(u,v),\chi>-{\cal E}^{c}(uv,\chi).
\end{eqnarray}

By (\ref{bv1v}) and (\ref{eee4}), we obtain the following
expression of the diffusion part $\mathcal{E}^{c}$.

\begin{thm}\label{thm07}
Suppose Assumption \ref{as33} holds. Let $u,v\in C_{0}(E)\cap D(\mathcal{E})$ and $\chi\in C_{0}(E)\cap D(\mathcal{E})$ satisfying $\chi=1$ on a neighbourhood of ${\rm supp}[u]\cup{\rm supp}[v]$. Then
$$
\mathcal{E}^{c}(u,v)=\frac{1}{2}\int_{E}\chi d\mu_{<u,v>}^c+<L(u,v),\chi>.
$$
\end{thm}

Similar to Theorem \ref{thm9}, we can derive the following transformation rules for $\mu^c_{<\cdot,\cdot>}$ and $L(\cdot,\cdot)$.

\begin{thm}\label{hhg}
Let $u, v, w, f\in C_{0}(E)\cap D(\mathcal{E})$. Then

\noindent (i) $d\mu^{c}_{<uv,w>}={u}d\mu^{c}_{<v,w>}+{v}d\mu^{c}_{<u,w>}$.

\noindent (ii) $<L(u,vw),f>=<L(u,v),wf>$.

\noindent (iii) $<L(uv,w),f>=<L(u,w),vf>+<L(v,w),uf>$.
\end{thm}

We use ${\cal F}_{loc}$ to denote the set of all functions $u$
such that for any relatively compact open set $V$ there exists a
$w\in C_{0}(E)\cap D(\mathcal{E})$ such that $u=w$ on $V$. Then,
by an argument similar to that given after the proof of Theorem
\ref{thm9}, we can extend   $\mu^c_{<u,v>}$ and $L(u,v)$ to
$u,v\in {\cal F}_{loc}$. The  transformation rules given in
Theorem \ref{hhg} still hold with $C_{0}(E)\cap D(\mathcal{E})$
replaced by ${\cal F}_{loc}$.

Now we make the following assumption.

\begin{assum}\label{jw1}
There exist a sequence of Dirichlet forms $({\cal Q}^l,
D({\cal Q}^l))$ on $L^2(\Omega_l;m)$ and a sequence of positive
constants $C_l$ such that $C_0(\Omega_l)\cap
D(\mathcal{E})=C_0(\Omega_l)\cap D({\cal Q}_l)$ and
$$
{\cal E}_1(g,g)\le C_l{\cal Q}^l_1(g,g), \ \ \forall g\in
C_0(\Omega_l)\cap D(\mathcal{E}).
$$
\end{assum}

\begin{thm}\label{mn1} Suppose Assumption \ref{jw1} holds and $J$ is a finite measure on $E\times E\backslash d$. Let $\Phi\in C^2(\mathbb{R}^m)$, $u_1,\dots,u_m,v,w\in {\cal F}_{loc}$ and $f\in C_{0}(E)\cap D(\mathcal{E})$. Then

\noindent (i) $d\mu^c_{<\Phi(u_1,\dots,u_m),v>}=\sum_{i=1}^m\Phi_{x_i}(u_1,\dots,u_m)d\mu^c_{<u_i,v>}$.

\noindent (ii) $<L(\Phi(u_1,\dots,u_m),vw),f>=\sum_{i=1}^m<L(u_i,v),\Phi_{x_i}(u_1,\dots,u_m)wf>$.
\end{thm}

The proof of Theorem \ref{mn1} is similar and simpler than that of
Theorem \ref{t28}. We omit the details here. We only point out
that \cite[(3.2.27)]{Fu94} and  Assumption \ref{jw1}  ensure the
convergence of $\Phi^{(k)}(u)$ (resp.
$\Phi^{(k)}_{x_i}(u)-\Phi^{(k)}_{x_i}(0)$) to $\Phi(u)$ (resp.
$\Phi_{x_i}(u)-\Phi_{x_i}(0)$) w.r.t. the $\tilde{\cal
Q}_{1}^{1/2}$-norm and hence the $\tilde{\cal E}_{1}^{1/2}$-norm,
and the finiteness of $J$ ensures that the dominated convergence
theorem can be applied directly. Theorems \ref{lll} and \ref{mn1}
will be applied in a forthcoming work to obtain the strong
continuity of generalized Feynman-Kac semigroups for Markov
processes associated with semi-Dirichlet forms.


\end{document}